\documentclass[a4paper,reqno]{amsart}

\usepackage[utf8]{inputenc}
\usepackage{amsthm, amssymb, amsmath, amsfonts}
\usepackage{url}
\usepackage{graphicx}
\usepackage{psfrag}

\newtheorem{thm}{Theorem}[section]
\newtheorem{prop}[thm]{Proposition}
\newtheorem{lem}[thm]{Lemma}

\newenvironment{rem}
{ %
\refstepcounter{thm} %
\noindent \textit{Remark \thesection.\arabic{thm}.}%
}{}

\renewcommand{\le}{\leqslant}
\renewcommand{\ge}{\geqslant}
\renewcommand{\subset}{\subseteq}
\newcommand{\mcl}{\mathcal}

\newcommand{\E}{\mathbb{E}}
\newcommand{\EE}{\mathbf{E}}

\newcommand{\N}{\mathbb{N}}

\newcommand{\1}{\mathbf{1}}
\newcommand{\R}{\mathbb{R}}
\newcommand{\Q}{\mathbb{Q}}
\newcommand{\Z}{\mathbb{Z}}
\renewcommand{\P}{\mathbb{P}}
\newcommand{\PP}{\mathbf{P}}

\newcommand{\ov}{\overline}
\newcommand{\td}{\tilde}
\newcommand{\eps}{\varepsilon}
\def\d{{\mathrm{d}}}

\newcommand{\mclC}{\mathcal{C}}
\newcommand{\ovmclC}{\ov{\mathcal{C}}}
\newcommand{\var}{\mathbb{V}\mathrm{ar}}
\newcommand{\mfk}{\mathfrak}
\renewcommand{\ll}{\left}
\newcommand{\rr}{\right}
\renewcommand{\phi}{\varphi}

\newcommand{\CC}{\texttt{CC}}
\newcommand{\hata}{\hat{a}}
\newcommand{\hatA}{\hat{A}}
\newcommand{\e}{\texttt{e}}
\renewcommand{\emptyset}{\varnothing}

\DeclareMathOperator{\m}{\mathsf{m}}

\title[Lyapunov exponents, shape theorems and large deviations]{Lyapunov exponents, shape theorems and large deviations for the random walk in random potential}
\author{Jean-Christophe Mourrat}

\address{Ecole polytechnique fédérale de Lausanne, institut de mathématiques, station 8, 1015 Lausanne, Switzerland.} 

\begin{document}
\begin{abstract}
We consider the simple random walk on $\Z^d$ evolving in a potential of independent and identically distributed random variables taking values in $[0,+\infty]$. We give optimal conditions for the existence of the quenched point-to-point Lyapunov exponent, and for different versions of a shape theorem. The method of proof applies as well to first-passage percolation, and builds up on an approach of \cite{cd}. The weakest form of shape theorem holds whenever the set of sites with finite potential percolates. Under this condition, we then show the existence of the quenched point-to-hyperplane Lyapunov exponent, and give a large deviation principle for the walk under the quenched weighted measure. 

\bigskip


\medskip


\end{abstract}
\maketitle
\tableofcontents
%
%
%
%
%
%
%
%
\section{Introduction}
\label{s:intro}
\setcounter{equation}{0}

Let $(S_n)_{n \in \N}$ be the simple random walk on $\Z^d$, $d \ge 2$. We write $\PP_x$ for the law of the random walk starting from position $x$, and $\EE_x$ for the associated expectation. Independently of $S$, we give ourselves a family $(V(x))_{x \in \Z^d}$ of independent and identically distributed random variables with values in $[0,+\infty]$, whose law we write $\P$ (and associated expectation $\E$), and that we call the \emph{potential}. To avoid degeneracy, we assume that $V(0)$ is not almost surely equal to $0$.

For $y \in \Z^d$, let us write $H_y$ for the entrance time of the walk at site $y$:
\begin{equation}
\label{defHy}
H_y = \inf \{ n \ge 0 : S_n = y \}.
\end{equation}
For any $x,y \in \Z^d$, we define
$$
e(x,y) = \EE_x\left[ \exp\left( -\sum_{n = 0}^{H_y - 1}  V(S_n) \rr) \ \1_{\{H_y < +\infty\}} \right] \qquad (e(x,y) = 1 \text{ if } x = y).
$$
The first question we want to address concerns the behaviour of $e(0,y)$ as $y$ tends to infinity. The question may be asked about the mean value of $e(0,y)$ (the \emph{annealed} situation), or about the typical value of $e(0,y)$ (the \emph{quenched} situation). The annealed behaviour has been successfully addressed by \cite{flu1} under the broadest possible assumption on the distribution of $V(0)$ --- that is, no assumption at all. From now on, we therefore focus on the quenched setting.
%
%
%
We will always assume that
\begin{itemize}
\item[(H)] $\P[V(0) < +\infty] > p_c$, where $p_c$ is the critical probability of site percolation in $\Z^d$.
\end{itemize}
This assumption guarantees the existence of an infinite connected component of sites $x$ with $V(x) < +\infty$, which is clearly a necessary condition to ensure that unbounded trajectories of the random walk really contribute to $e(x,y)$ for a fixed potential. We denote this infinite connected component by $\mathbf{C}_\infty$.

Let us define
\begin{equation}
\label{defa}
a(x,y) = - \log e(x,y).	
\end{equation}
The quantity $a(x,y)$ can be thought of as measuring the cost of travelling from $x$ to $y$ for the random walk in the potential $V$, and is analogous to the travel time of first-passage percolation. For $x,y \in \Z^d$, we write $x\sim y$ if $\|y-x\|_2 = 1$ (where $\|\cdot \|_2$ is the Euclidian norm), and define
\begin{equation}
\label{defZ}
Z(x) = \min_{y \sim x} V(y).
\end{equation}
\begin{thm}[Point-to-point Lyapunov exponent]
\label{t:existpoint}	
Let $x \in \Z^d \setminus \{0\}$. As $n$ tends to infinity, the quantity
\begin{equation}
\label{eq:limpoint}
\frac{1}{n} a(0,nx)
\end{equation}
\begin{enumerate}
\item converges in $L^1$ if and only if $\E[V(0)] < +\infty$ ;
\item converges almost surely if and only if $\E[Z(0)] < +\infty$ ;
\item converges in probability if and only if $V(0) < + \infty$ a.s.
\end{enumerate}
In any of these cases, the limit $\alpha(x)$ is deterministic, and $\alpha$ can be extended into a norm on $\R^d$. When $V(0) < + \infty$ a.s., we also have
$$
\liminf_{n \to + \infty}  \frac{1}{n} a(0,nx) = \alpha(x) \quad \text{ almost surely}.
$$
\end{thm}
In \cite[Proposition~4]{zer}, it is proved that $\E[V(0)] < +\infty$ is a sufficient condition for $L^1$ and almost sure convergence. It is easy to check that the integrability of $V(0)$ is also necessary for $L^1$ convergence, as $a(0,x) \ge V(0)$ for any $x \neq 0$, so part (1) is in fact already known. 

Let us discuss part (3) of the theorem. For two-dimensional first-passage percolation, \cite{cd} obtained similar results using a strategy that we can informally describe in our present context. Assuming that $V(0) < + \infty$ a.s.,\ we can choose $M$ very large so that there exists an infinite connected component $\mathbf{C}'_\infty$ of sites $x$ such that $V(x) \le M$, with only very small ``holes'', i.e.\ such that the complement of $\mathbf{C}'_\infty$ has only small connected components. Then, we encapsulate each site $x$ within a small set $\mathbf{C}_x$ such that any point on the boundary of $\mathbf{C}_x$ lies in $\mathbf{C}'_\infty$, and define an approximate cost function $\hata(x,y)$, that measures the minimal cost for a travel from a point of the boundary of $\mathbf{C}_x$ to the set $\mathbf{C}_y$. With some care, one can apply the subadditive ergodic theorem to $\hata(x,y)$, and then justify that as far as convergence in probability of (\ref{eq:limpoint}) is concerned, $\hata(x,y)$ is indeed a good approximation of $a(x,y)$.

We can list the main differences between \cite{cd} and our present context.
\begin{enumerate}
\item Here, the ``travel time'' $a(x,y)$ is defined as a weighted average over the measure of the simple random walk, 
\item the randomness is attached to sites of $\Z^d$ and not to edges,
\item we do not restrict ourselves to $d = 2$.
\end{enumerate}

Difference (1) does not have any true influence over the methods employed either by \cite{cd} or here. Difference (2) introduces some difficulty when proving part (2) of Theorem~\ref{t:existpoint}, forcing us to introduce a second approximation of $a(x,y)$, but the argument is not different in nature. Some care is also needed to adapt contour-type arguments to higher-dimensional situations, but this problem was already addressed in \cite{kes} for first-passage percolation.

After these adaptations, the strategy works fine. We want however to pursue our investigation to shape theorems, existence of point-to-hyperplane exponents, and large deviation principle. They could be derived using the strategy we just sketched, but the minimal requirement for these results should no longer be that $V(0) < + \infty$ a.s., but instead assumption (H). For this reason, although it is not necessary for the purpose of proving Theorem~\ref{t:existpoint}, we want to build an approximation of $a(x,y)$ that is efficient in the full range of this assumption. The most important difference between what we will do and \cite{cd} is thus that
\begin{enumerate}
\item[(4)] we allow the potentials to take the value $+ \infty$.
\end{enumerate}
The strategy sketched rests on the fact that one can define an efficient approximation of $a(x,y)$ provided one can find $M$ large enough so that $\P[V(0) > M]$ is smaller than some strictly positive constant. It would thus be a method robust enough to handle the case when $\P[V(0) =+\infty]$ is sufficiently small. 

In order to cover the full range of assumption (H), we will rely on a renormalization procedure, that in a sense enables to push the probability of appearance of problematic points (those with $V(x) > M$) 
as close as we wish to $0$. We can then define the approximate $\hata(x,y)$ in the renormalized scale in the spirit of the above informal description. We will show that there exists a norm $\alpha$ on $\R^d$ such that under (H) only, the function $\hata$ thus constructed satisfies, for any $x \in \Z^d$,
\begin{equation}
\label{defalpha}
\frac{1}{n} \hata(0,nx) \xrightarrow[n \to +\infty]{\text{a.s.}} \alpha(x).	
\end{equation}
We call $\alpha$ the \emph{Lyapunov norm} associated with the potential $V$. We obtain part (3) of Theorem~\ref{t:existpoint} by showing that when $V(0) < + \infty$ a.s.,\ $\hata$ is sufficiently good an approximation of $a$  for the convergence in probability of (\ref{eq:limpoint}) to be preserved.

Now equipped with this robust approximation of $a(x,y)$, we can go further and describe the limit shape of the set
\begin{equation}
\label{defAt}
A_t = \{x \in \Z^d : a(0,x) \le t \}.
\end{equation}
We write
\begin{equation}
\label{defAtcirc}
A_t^\circ = A_t + [0,1)^d.
\end{equation}
\begin{thm}[Shape theorems]
\label{t:shape}
Let $\alpha$ be the norm on $\R^d$ such that (\ref{defalpha}) holds for any $x \in \Z^d$, and 
\begin{equation}
\label{defK}
K = \{ x \in \R^d : \alpha(x) \le 1 \}.	
\end{equation}
\begin{enumerate}
\item One has
\begin{equation}
\label{eq:shapemoment}
\forall \eps > 0, \quad \P[(1-\eps)K \subset t^{-1} A_t^\circ \subset (1+\eps)K \text{ for all } t \text{ large enough} ]=  1	
\end{equation}
if and only if $\E[Z(0)^d] < + \infty$.
\item One has
\begin{equation}
\label{shapemeas}
| (t^{-1} A_t^\circ) \ \triangle \ K | \xrightarrow[t \to + \infty]{\text{a.s.}} 0
\end{equation}
if and only if $V(0) < + \infty$ a.s., where in (\ref{shapemeas}), $|\cdot|$ denotes the Lebesgue measure and $\triangle$ the symmetric difference.
\item Under assumption (H) only, with probability one, 
\begin{equation}
\label{eq:pth1}
\forall \eps > 0, \quad t^{-1} A_t \subset (1+\eps)K \text{ for all } t \text{ large enough},
\end{equation}
and moreover, for any function $(\texttt{e}_t)_{t \ge 0}$ such that $\lim_{t \to \infty} \e_t = + \infty$  and $\lim_{t \to \infty} \e_t/t = 0$, on the event $0 \in \mathbf{C}_\infty$, one has
\begin{equation}
\label{eq:pth2}
| (t^{-1} A_t^{\texttt{e}_t}) \ \triangle \ K | \xrightarrow[t \to +\infty]{\text{a.s.}} 0 ,
\end{equation}
where, for a set $A \subset \R^d$ and any $\texttt{e} > 0$, we write $A^{\texttt{e}}$ for the $\texttt{e}$-enlargement of $A$, that is, 
$$
A^{\texttt{e}} = \{ x \in \R^d : \exists y \in A \text{ s.t. } \|x-y\|_2 \le \texttt{e} \}.
$$
\end{enumerate}
\end{thm}
\begin{rem}
\label{rem1}
One can check that statement (\ref{eq:pth1}) is equivalent to the fact that
$$
\liminf_{\|x\| \to + \infty} \frac{a(0,x)}{\alpha(x)} \ge 1.
$$
Similarly, the fact that, for any $\eps > 0$, $(1-\eps)K \subset t^{-1} A_t$ for all $t$ large enough is equivalent to
$$
\limsup_{\|x\| \to + \infty} \frac{a(0,x)}{\alpha(x)} \le 1.
$$
As a consequence, part (1) of Theorem~\ref{t:shape} extends \cite[Theorem~8]{zer}, where (\ref{eq:shapemoment}) was obtained assuming that $\E[V(0)^d]$ is finite. 
\end{rem}

\medskip

\begin{rem}
\label{r:unifst}
Assuming further that $V(0)$ is uniformly bounded away from $0$ if $d = 2$, \cite[Theorem~13]{zer} provides a \emph{uniform} shape theorem, which gives the asymptotic behaviour of $a(x,y)$ when both $x$ and $y$ are allowed to go to infinity simultaneously, with $\|y-x\| \to + \infty$.
\end{rem}

\medskip

\begin{rem}
Statements similar to Theorems~\ref{t:existpoint} and~\ref{t:shape} can be derived with $a(x,y)$ replaced by the logarithm of the Green function, using for instance \cite[Lemma~5]{zer}.
\end{rem}

\medskip

\begin{rem}
In part (3) of Theorem~\ref{t:shape}, the assumption that $\lim_{t \to \infty} \e_t = +\infty$ is necessary if $V(0)$ can take the value $+\infty$. Indeed, in this case, the set $A_t$ is riddled with holes of arbitrarily large size. To see this, consider an ``island'' in $\Z^d$ surrounded by sites where the potential is infinite. Clearly, $a(0,x)$ is infinite for any site $x$ of this island (provided $0$ is not itself in the island), so these sites do not belong to $A_t$, and moreover, copies of this island occupy a non-zero proportion of the space. Taking the island large enough, one can check that (\ref{eq:pth2}) would not hold if $\e_t$ was to remain bounded.
\end{rem}

\medskip

We now turn to point-to-hyperplane exponents. 
For $x \in \R^d$ and $t \ge 0$, let $H^*_{x,t}$ be the stopping time
$$
H_{x,t}^* = \inf \{n \in \N : S_n \cdot x \ge t \},
$$
and
$$
a^*(x,t) = - \log \EE_0\left[ \exp\left( -\sum_{n = 0}^{H_{x,t}^* - 1}  V(S_n) \rr) \ \1_{\{H_{x,t}^* < +\infty\}} \right].
$$
We define the dual norm to $\alpha$ as
\begin{equation}
\label{defalpha*}
\alpha^*(x) = \sup_{y \neq 0} \ \frac{x \cdot y}{\alpha(y)}.	
\end{equation}
\begin{thm}[Point-to-hyperplane Lyapunov exponent]
\label{t:existplane}	
We recall that assumption (H) holds. For any $x \in \R^d \setminus \{0\}$, on the event $0 \in \mathbf{C}_\infty$, one has
\begin{equation}
\label{eq:pth}
 \frac{1}{t} a^*(x,t) \xrightarrow[t \to +\infty]{\text{a.s.}} \frac{1}{\alpha^*(x)}.	
\end{equation}
\end{thm}
This result was previously obtained in \cite[Corollary~C]{flu1} assuming that $\E[V(0)^d]$ is finite. Here, we obtain Theorem~\ref{t:existplane} as a consequence of part (3) of Theorem~\ref{t:shape}.

\medskip

\begin{rem}
\label{remiv}
In \cite{iv}, the authors consider potentials that are uniformly bounded away from $0$, and show that if $\P[V(0) = + \infty]$ is small enough, $d \ge 4$, $x$ is parallel to a coordinate axis, and at sufficiently high temperature (see the paragraph ``related works'' below for a definition of this), then the limit in (\ref{eq:pth}) exists, and coincides with the annealed point-to-hyperplane exponent on the event $0 \in \mathbf{C}_\infty$. 
\end{rem}

\medskip

\begin{rem}
Theorems~\ref{t:existpoint}, \ref{t:shape} and \ref{t:existplane} have their counterparts for first-passage percolation, that are commented on in section~\ref{s:firstpass}.
\end{rem}

\medskip

We then study the large deviations of the random walk under the \emph{weighted measure} $\PP_{n,V}$ defined by
$$
\frac{\d \PP_{n,V}}{\d \PP_0} = \frac{1}{Z_{n,V}} \ \exp\ll( -\sum_{k = 0}^{n-1} V(S_k) \rr),
$$
where $Z_{n,V}$ is a normalizing constant that may be called the \emph{partition function}, and is given by
\begin{equation}
\label{defZVn}
Z_{n,V} = \EE_0\ll[\exp\ll( -\sum_{k = 0}^{n-1} V(S_k) \rr) \rr].	
\end{equation}
\begin{thm}[Large deviations]
\label{t:ldp}
Assume that $0$ is in the support of the distribution of $V(0)$ (and that (H) is satisfied). For any $\lambda \ge 0$, let $\alpha_{\lambda}$ be the Lyapunov norm associated with the potential $V_\lambda = \lambda + V$, and, for any $x \in \R^d$, let
\begin{equation}
\label{defI}
I(x) = \sup_{\lambda \ge 0} (\alpha_\lambda(x) - \lambda).	
\end{equation}
Almost surely, on the event $0 \in \mathbf{C}_\infty$, the rescaled random walk $S_n/n$ under $\PP_{n,V}$ satisfies a large deviation principle with speed $n$ and (convex lower semi-continuous good) rate function $I$. In other words, with probability one, the fact that $0 \in \mathbf{C}_\infty$ implies that the following two properties hold:
\begin{equation}
\label{ldclosed}
\begin{array}{l}
\text{for any closed } F \subset \R^d, \\
\displaystyle{\liminf_{n \to + \infty} \ -\frac1n \log \PP_{n,V}[S_n \in n F] \ge \inf_{F} I},
\end{array}	
\end{equation}
and
\begin{equation}
\label{ldopen}
\begin{array}{l}
\text{for any open } O \subset \R^d, \\
\displaystyle{\limsup_{n \to + \infty} \ -\frac1n \log \PP_{n,V}[S_n \in n O] \le \inf_{O} I}.
\end{array}	
\end{equation}
\end{thm}

This result was obtained in \cite[Theorem~D]{zer} assuming that $\E[V(0)^d]$ is finite. The proof given there relies on the uniform shape theorem mentioned in Remark~\ref{r:unifst}. In order to obtain Theorem~\ref{t:ldp}, we will show a much weaker (but still sufficient) form of uniform shape theorem, in which we control $\hata(x,y)$ only when $x$ and $y$ tend to infinity simultaneously along the same fixed direction (see Proposition~\ref{alllambda}).

\medskip

\begin{rem}
Applying Varadhan's lemma, one can deduce from Theorem~\ref{t:ldp} a large deviation principle for the random walk with a constant drift, under the weighted measure. One of the most remarkable consequences of this large deviation principle is that there is a transition from sub-ballistic to ballistic behaviour according to whether the drift $h$ satisfies $\alpha^*(h) < 1$ or $\alpha^*(h) > 1$. We refer to \cite[Theorem~D]{flu1} for a precise statement. There, it is assumed that $\E[V(0)^d]$ is finite, as a necessary condition for the large deviation principle obtained in \cite{zer} to hold. Using instead Theorem~\ref{t:ldp} as an input, the proof given there applies as well, so that \cite[Theorem~D]{flu1} is in fact valid under assumption (H) only, on the event $0 \in \mathbf{C}_\infty$.
\end{rem}

\medskip

\noindent \textbf{Related works.} Outside of the previously mentioned papers \cite{zer,flu1}, one should also note recent work on the question of whether the disorder is \emph{weak} or \emph{strong}. The disorder is said to be weak if quenched and annealed Lyapunov exponents coincide. The potential is usually taken of the form $\lambda + \beta V$, for $\lambda,\beta \ge 0$, where $\beta$ is interpreted as the inverse temperature. Weak disorder results have been obtained by \cite{flu2,zyg1,iv} assuming $\lambda > 0$, $d \ge 4$, and $\beta$ sufficiently small (that is, in the high temperature regime). On the other hand, in \cite{zyg2}, strong disorder has been proved to hold for any $\lambda > 0$ and $\beta > 0$  if $d \le 3$, and for sufficiently large $\beta$ in any dimension. Except for \cite{iv} that was mentioned in Remark~\ref{remiv}, these works rely on the result of existence of Lyapunov exponents given by \cite{zer}. We refer to \cite{ivreview} for a survey on this and related questions.

In a slightly different direction, \cite{wan1,wan2,kmz} study the asymptotic behaviour of the quenched and annealed Lyapunov exponents under the potential $\beta V$, in the limit of high temperatures. Assuming that the potential is integrable, it is shown in \cite{kmz} that both the quenched and annealed Lyapunov exponents are equivalent to the same $c \sqrt{\beta}$ as $\beta$ tends to $0$, where $c$ is an explicit constant.

The questions we consider here have been first investigated for Brownian motion in a potential that can be written as 
$\sum_{i} W(\cdot -x_i)$, where $W$ is a positive, bounded, measurable and compactly supported function, and $(x_i)$ is a Poisson point process of fixed intensity on $\R^d$. For this model, existence of Lyapunov exponents, shape theorem and large deviation principle were obtained in \cite{sznit}. Naturally, many ideas originating from \cite{sznit} (or equivalently, from \cite[Chapter~5]{sznitman}) permeate throughout the present paper.

\medskip 

\noindent \textbf{Organization of the paper.} Sections~\ref{s:ptpmoment} and~\ref{s:shapemoment} are devoted, respectively, to the proof of part (2) of Theorem~\ref{t:existpoint}, and part (1) of Theorem~\ref{t:shape}. There, we introduce another approximation of $a(x,y)$ than the one discussed in the introduction. These sections also serve as a preparation for the similar but more intricate situation encountered in the rest of the paper. After recalling some facts about discrete geometry in section~\ref{s:geom}, we tackle the construction of the renormalization procedure in section~\ref{s:renormalize}. We then obtain the proof of part (3) of Theorem~\ref{t:existpoint} and parts (2-3) of Theorem~\ref{t:shape} in sections~\ref{s:ptpnomoment} and~\ref{s:shapenomoment}, respectively. The next two sections are devoted to the proof of the large deviation principle. Section~\ref{s:ldplow} provides the lower bound (\ref{ldclosed}), while section~\ref{s:ldpup} gives a proof of the more difficult upper bound (\ref{ldopen}). Finally, section~\ref{s:firstpass} discusses counterparts of Theorems~\ref{t:existpoint}, \ref{t:shape} and \ref{t:existplane} in the context of first-passage percolation.

\medskip

\noindent \textbf{Notations.} For $x \in \R^d$ and $r > 0$, let
\begin{equation}
\label{defDxr}
D(x,r) = \{y \in \R^d : \|y-x\|_1 \le r\}
\end{equation}
be the $\|\cdot\|_1$-ball of centre $x$ and radius $r$. For $x \in \Z^d$ and $R \in \N$, we call $\emph{box}$ of centre $x$ and size $R$ the set defined by
\begin{equation}
\label{defBxr}
B(x,R) = x + \{-R,\ldots,R\}^d.	
\end{equation}

%
%
%
%
%
%
%
\section{Point-to-point exponent under a moment condition}
\label{s:ptpmoment}
\setcounter{equation}{0}
The aim of this section is to prove part (2) of Theorem~\ref{t:existpoint}.

First, let us see that in order for almost sure convergence to hold, one must have $\E[Z(0)] < +\infty$. If this is not the case, a Borel-Cantelli argument shows that for any $M > 0$, the events
$$
Z(nx) \ge Mn
$$
happen for infinitely many $n$'s with probability one. Since $a(0,nx) \ge Z(nx)$, this observation implies that almost surely,
$$
\limsup_{n \to +\infty} \frac1n a(0,nx) = + \infty.
$$

We will now show that the condition $\E[Z(0)] < +\infty$ is sufficient for almost sure convergence. The problem we face is that $a(x,y)$ need not be integrable. We need to define an approximation of it that will, at the same time, satisfy the assumptions of the ergodic theorem, and be sufficiently close to $a(x,y)$ for the almost sure convergence of (\ref{eq:limpoint}) to be preserved.

For any $x \in \Z^d$, we define $\m(x)$ to be the $y \sim x$ such that $V(y) = Z(x)$, where $Z(x)$ is defined in (\ref{defZ}), and with some deterministic tie-breaking rule. Let $a_m(x,y)$ be the cost associated to a travel from $\m(x)$ to $\m(y)$, that is to say,
\begin{equation}
\label{defam}
a_m(x,y) = a(\m(x),\m(y)).
\end{equation}
Before explaining how $a_m(x,y)$ approximates $a(x,y)$, we need the following lemma.
\begin{lem}
\label{2dpaths}
For any $x \in \Z^d$, there exist $2d$ paths linking $0$ to $x$ and such that, outside of the starting and ending points, the sets of sites visited by any two of these paths do not intersect. Moreover, one can choose the paths so that the length of the longest one does not exceed $\|x\|_1 + 8$. 
\end{lem}
\begin{proof}
By symmetry we can assume that all coordinates of $x$ are positive ($\ge 0$), and prove the result by induction on the dimension. In order to control the lengths, it will be convenient to impose all the paths we construct to be contained in the set
\begin{equation}
\label{constraintpath}
\mcl{H}_{-1} = \{x = (x_1,\ldots, x_d) \in \Z^d : x_d \ge -1 \}.	
\end{equation}

In dimension $2$, if $x = (n,0)$ for some $n$, then one can travel through
\begin{itemize}
\item $(0,0) \to (n,0)$,
\item $(0,0) \to (0,1) \to (n,1) \to (n,0)$,
\item $(0,0) \to (0,-1) \to (n,-1) \to (n,0)$,
\item $(0,0) \to (-1,0) \to (-1,2) \to (n+1,2) \to (n+1,0) \to (n,0)$ (this path has length $\|x\|_1 + 8$, the others being shorter),
\end{itemize}
where the arrow denotes travel by a straight line. If $x = (0,n)$, we can obtain $4$ such paths by symmetry (although not by a rotation for we want them to stay in the set $\mcl{H}_{-1}$). If $x = (m,n)$ with both $m$ and $n$ strictly positive, then the two rectangles that are the boundaries of
$$
[0,m+1] \times [-1,n] \quad \text{ and } \quad [-1,m] \times [0,n+1]
$$
intersect only at $(0,0)$ and $(m,n)$. The sides of the rectangles thus define an adequate set of $4$ paths.

Let us now assume that the lemma is true in dimension $d$, and consider $x = (x_1,\ldots,x_{d+1}) \in \Z^{d+1}$. One can find a set of $2d$ adequate paths linking $0$ to $\td{x} = (x_1,\ldots,x_{d},0)$ in the hyperplane $\mcl{H}_d$ made of points with the $(d+1)$-th coordinate equal to $0$. If $x_{d+1} = 0$, then we can safely add the paths
\begin{itemize}
\item $(0,\ldots,0,1) \to (x_1,\ldots,x_{d},1) \to (x_1,\ldots,x_{d},0)$,
\item $(0,\ldots,0,-1) \to (x_1,\ldots,x_{d},-1) \to (x_1,\ldots,x_{d},0)$.
\end{itemize}
Otherwise, we prune the last step of each of the $2d$ paths, so that we are left with $2d$ paths connecting $0$ to each neighbour of $\td{x}$ that lies in $\mcl{H}_d$. For each path, we then travel in the direction orthogonal to $\mcl{H}_d$ up to the height $x_{d+1}$, thus connecting $0$ to a neighbour of $x$ lying in the hyperplane $x + \mcl{H}_d$, and we add a last step towards $x$ to each of these paths. When $(x_1,\ldots, x_d,0)$ is not a neighbour of $0$, this construction leaves the following two paths available
\begin{itemize}
\item $0 \to (0,\ldots,0,-1) \to  (x_1,\ldots, x_d,-1) \to x$
\item $0 \to (0,\ldots,0,x_{d+1} + 1) \to  (x_1,\ldots, x_d,x_{d+1} + 1) \to x$.
\end{itemize}
Otherwise, say for definiteness that $(x_1,\ldots, x_d,0) = (0,\ldots,0,1,0)$. Then the arrow pointing in the $(d+1)$-th dimension ``upwards'' from $0$ is already used, but one can use instead the path
\begin{itemize}
\item $0 \to (0,\ldots,0,1,0) \to x$.
\end{itemize}
Note that if $y = (y_1,\ldots,y_d,0) \in \mathcal{H}_d$ is visited by one of the paths that is already constructed, then it must be that $y_d \ge -1$. As a consequence, the path
\begin{itemize}
\item $0 \to (0,\ldots,0,-1) \to (0,\ldots,0,-2,-1) \to (0,\ldots,0,-2,x_{d+1} + 1)$ 
\begin{flushright}$\to (x_1,\ldots,x_d,x_{d+1} + 1) \to x$.\end{flushright}
\end{itemize}
is still available (and goes through $\|x\|_1 + 8$ edges), so the proof is complete.
\end{proof}
For a path $\gamma = (\gamma_0,\ldots,\gamma_l)$, we write $|\gamma| = l$ for the length of the path, and we understand the notation ``$x \in \gamma$'' to mean that there exists $k \le l$ such that $x = \gamma_k$. We also write $\td{\gamma}$ to denote the path with starting and ending points removed. For any $x \in \Z^d$, Lemma~\ref{2dpaths} gives us $2d$ paths $\gamma^1_{0,x}, \ldots, \gamma^{2d}_{0,x}$ that connect $0$ to $x$ and do not intersect each other, except at the starting and ending points. 

\begin{equation}
\label{defgammaxy}
\begin{array}{l}
\text{For any } x,y \in \Z^d \text{ and } 1\le i \le 2d,\text{ we let }\gamma^i_{x,y} \text{ be } \\
\text{the translation by } x \text{ of the path } \gamma^i_{0,y-x}.
	\end{array}	
\end{equation}
The path $\gamma^i_{x,y}$ thus connects $x$ to $y$. The following proposition gives a precise meaning to the idea that $a_m$ is an approximation of $a$. 
\begin{prop}
\label{approxam}
Let 
\begin{equation}
\label{defum}
u_m(y) = \sum_{\substack{y' \sim y \\ y'' \sim y}} \min_{1 \le i \le 2d} \sum_{z \in \td{\gamma}^i_{y',y''}} (V(z) + \log(2d)).
\end{equation}	
For any $x,y \in \Z^d$, one has
\begin{equation}
\label{eq:approxam1}
a_m(x,y)  \le a(x,y) + Z(x) + 2 \log(2d) + u_m(y),
\end{equation}
\begin{equation}
\label{eq:approxam2}
a(x,y) \le a_m(x,y) + V(x) + Z(y) + 2 \log(2d).
\end{equation}
Moreover, for any $x,y,z \in \Z^d$, one has
\begin{equation}
\label{triangleam}
a_m(x,z) \le a_m(x,y) + a_m(y,z),
\end{equation}
and the same is true if in (\ref{triangleam}), $a_m$ is replaced by $a$.
\end{prop}
\begin{rem}
\label{presquepreuve}
Note that the ``cost'' associated to the event that the random walk goes from $x$ to a given neighbour in one step is $V(x) + \log(2d)$, where $\log(2d)$ comes from the fact that the walk has probability $1/(2d)$ to make its first jump to this particular neighbour. Inequality (\ref{eq:approxam2}) can thus easily be understood. It states that the cost to travel from $x$ to $y$ is smaller than the cost to go from $x$ to $\m(x)$ in one step ($V(x) + \log(2d)$), plus the cost to go from $\m(x)$ to $\m(y)$ ($a_m(x,y)$), plus the cost to go from $\m(y)$ to $y$ in one step ($Z(y) + \log(2d)$). A similar interpretation holds for (\ref{triangleam}). Inequality (\ref{eq:approxam1}) is similar, except that multiple possible paths are considered.	
\end{rem}

\medskip

\begin{rem}
In view of the way we chose the paths in (\ref{defgammaxy}), it is clear that $(u_m(y))_{y \in \Z^d}$ are identically distributed random variables. In the definition of $u_m(y)$ (see (\ref{defum})), taking the minimum over disjoint paths will guarantee us good integrability properties for these random variables.
\end{rem}

\begin{proof}[Proof of Proposition~\ref{approxam}]
We focus on proving inequality (\ref{eq:approxam1}). The inequality is obvious if $x=y$, so we assume the contrary. In order to lighten the notation slightly, we will write $\1_h(y) = \1_{\{ H_y < + \infty \}}$. 

The important observation is that in order to reach $y$ from the point $x$, the walk must travel through a neighbour of $y$ first. Letting $\tau(y)$ be the first $n$ such that $S_n \sim y$, we have
\begin{eqnarray}
\label{debut}
e^{-a(x,y)} & = & \EE_x\left[ \exp\left( -\sum_{n = 0}^{H_y - 1}  V(S_n) \rr) \ \1_h(y) \right] \notag \\
& \le & \sum_{y' \sim y} \EE_x\left[ \exp\left( -\sum_{n = 0}^{\tau(y)}  V(S_n) \rr) , \ \tau(y) < +\infty , \ S_{\tau(y)} = y'\right] .
\end{eqnarray}
Let us write $\mcl{A}_{y'}$ for the event $\{\tau(y) < +\infty , \ S_{\tau(y)} = y'\}$. We now want to go from $y'$ to $\m(y)$, but we must carefully choose along which path we will do so. Let $i \in \{1,\ldots 2d\}$ be such that
$$
\sum_{z \in \td{\gamma}^i_{y',\m(y)}} (V(z) + \log(2d))
$$
is minimal. Let us write $\mcl{B}_{y'}$ for the event that $\mcl{A}_{y'}$ is satisfied and the walk follows the path $\gamma^i_{y',\m(y)}$ starting from time $\tau(y)$, and note that
\begin{eqnarray*}
&& \EE_x\left[ \exp\left( -\sum_{n = 0}^{H_{\m(y)} - 1}  V(S_n) \rr) \ \1_h(\m(y)), \mcl{A}_{y'} \right]  \\
&& \qquad \ge \EE_x\left[ \exp\left( -\sum_{n = 0}^{\tau(y)}  V(S_n)  - \sum_{z \in \td{\gamma}^i_{y',\m(y)}} V(z) \rr), \mcl{B}_{y'} \right] \\
&& \qquad \ge \EE_x\left[ \exp\left( -\sum_{n = 0}^{\tau(y)}  V(S_n)  - \sum_{z \in \td{\gamma}^i_{y',\m(y)}} V(z) \rr), \mcl{A}_{y'} \right] \ \ll(\frac{1}{2d}\rr)^{|\gamma^i_{y',\m(y)}|} \\
&& \qquad \ge \frac{e^{-u_m(y)}}{2d} \EE_x\left[ \exp\left( -\sum_{n = 0}^{\tau(y)}  V(S_n) \rr), \mcl{A}_{y'} \right].
\end{eqnarray*}
Summing over $y' \sim y$ and using (\ref{debut}), we thus obtain
\begin{equation}
\label{milieu}
\EE_x\left[ \exp\left( -\sum_{n = 0}^{H_{\m(y)} - 1}  V(S_n) \rr) \ \1_h(\m(y)) \right] \ge \frac{e^{-u_m(y)}}{2d} e^{-a(x,y)}.
\end{equation}
The l.h.s.\ of (\ref{milieu}) is very close to being $e^{-a_m(x,y)}$, so we are almost done. Let us write $(\Theta_t)_{t \in \N}$ for the time translations acting on the space of trajectories of the random walk, so that $(\Theta_t S)_n = S_{n+t}$. 
We have
\begin{eqnarray*}
e^{-a_m(x,y)} & = & \EE_{\m(x)}\left[ \exp\left( -\sum_{n = 0}^{H_{\m(y)} - 1}  V(S_n) \rr) \ \1_h(\m(y)) \right] \\
& \ge & \EE_{\m(x)}\left[ \exp\left( -\sum_{n = 0}^{H_{\m(y)} - 1}  V(S_n) \rr) \ \1_{\{S_1 = x\}}  \1_h(\m(y)) \circ \Theta_{1} \right].
\end{eqnarray*}
Using the Markov property at time $1$, the latter becomes
\begin{equation}
\label{eqvmx}
\underbrace{e^{-V(\m(x))} \PP_{\m(x)}[S_1 = x]}_{= {e^{-Z(x)}}/{2d}} \EE_{x}\left[ \exp\left( -\sum_{n = 0}^{H_{\m(y)} - 1}  V(S_n) \rr) \  \1_h(\m(y)) \right] .
\end{equation}
Inequality (\ref{eq:approxam1}) then follows from (\ref{milieu}) and (\ref{eqvmx}). The proofs of inequalities (\ref{eq:approxam2}) and (\ref{triangleam}) are simpler, and are based on Remark~\ref{presquepreuve}.
\end{proof}

We now want to show that $a_m(x,y)$ has nice integrability properties. This starts with a lemma.
\begin{lem}
\label{integrpath}
Let $\gamma^1,\ldots, \gamma^{2d}$ be finite paths visiting pairwise disjoint sets of points. If $Z(0)$ has finite $p$-th moment for some $p > 0$, then so does
$$
\min_{1 \le i \le 2d} \ \sum_{x \in \gamma^i} V(x).
$$
\end{lem}
\begin{proof}
The random variables $\ll( \sum_{x \in \gamma^i} V(x) \rr)_{1 \le i \le 2d}$ being independent, we have, for any $t > 0$,
\begin{equation}
\label{eq:pmin}	
\P \ll[ \ll( \min_{1 \le i \le 2d} \ \sum_{x \in \gamma^i} V(x) \rr)^p \ge t \rr] = \prod_{1 \le i \le 2d} \P \ll[ \sum_{x \in \gamma^i} V(x) \ge t^{1/p} \rr].
\end{equation}
Let $l_i$ be the number of sites visited by $\gamma^i$, and $l = \max_i l_i$. Then for any $i$,
$$
\P \ll[ \sum_{x \in \gamma^i} V(x) \ge t^{1/p} \rr] \le l \P \ll[V(0) \ge \frac{t^{1/p}}{l} \rr].
$$
The probability in (\ref{eq:pmin}) is thus smaller than
$$
l^{2d} \P[V(0) \ge t^{1/p}/l]^{2d} = l^{2d}  \prod_{z \sim 0} \P[V(z) \ge t^{1/p}/l] = l^{2d}  \P[Z(0)^p \ge t/l^p],
$$
and the lemma is proved by integration over $t$.
\end{proof}
\begin{prop}
\label{integram}
If $Z(0)$ has finite $p$-th moment for some $p > 0$, then so do $a_m(x,y)$ and $u_m(y)$, for any $x,y \in \Z^d$. 
\end{prop}
\begin{proof}
Let $\gamma$ be a nearest-neighbour path from $\m(x)$ to $\m(y)$. By forcing the random walk to follow the path $\gamma$, we obtain that
\begin{eqnarray*}
a_m(x,y) & \le &  \sum_{z \in \gamma \setminus \{\m(y)\}} (V(z) + \log(2d))	\\
& \le & \log(2d) + V(\m(x)) + \sum_{z \in \td{\gamma}} (V(z) + \log(2d)),
\end{eqnarray*}
with $V(\m(x)) = Z(x)$. Lemma~\ref{2dpaths} gives us $2d$ nearest-neighbour paths $\gamma^1,\ldots,\gamma^{2d}$ linking $\m(x)$ to $\m(y)$, and such that the sets of points visited by $\td{\gamma}^1, \ldots, \td{\gamma}^{2d}$ are pairwise disjoint. We have
\begin{equation}
\label{amxyle}
a_m(x,y) \le \log(2d) + Z(x) + \min_{1 \le i \le 2d} \sum_{z \in \td{\gamma}^i} (V(z) + \log(2d)),	
\end{equation}
so the result follows from Lemma~\ref{integrpath}. The integrability properties of the random variable $u_m(y)$ are handled in the same way. 
\end{proof}
We can now conclude this section.
\begin{proof}[Proof of part (2) of Theorem~\ref{t:existpoint}]
We assume that $Z(0)$ is integrable, and let $x \in \Z^d \setminus \{0\}$. From Proposition~\ref{integram} and inequality~(\ref{triangleam}), we learn that, in the words of \cite{kingman}, the doubly indexed sequence $(a_m(nx,px))_{n < p}$ is a stationary and integrable subadditive process. It thus follows that
\begin{equation}
\label{limam}
\frac1n a_m(0,nx) \text{ converges a.s. and in } L^1 \text{ to } \inf_{n \in \N} \frac1n {\E[a_m(0,nx)]}.
\end{equation}
Let us write $\alpha(x)$ for this limit. Proposition~\ref{approxam} yields that
\begin{equation}
\label{limsupam}
\limsup_{n \to +\infty} \frac1n a(0,nx) \le \alpha(x) + \limsup_{n \to +\infty} \frac1n {Z(nx)}.	
\end{equation}
That the limsup appearing in the r.h.s.\ of (\ref{limsupam}) is equal to $0$ a.s.\ follows from the fact that $(Z(nx))$ are identically distributed and integrable random variables, through a Borel-Cantelli argument. The liminf is handled in the same way using the other inequality of Proposition~\ref{approxam}, together with the fact ensured by Proposition~\ref{integram} that $(u_m(nx))$ are identically distributed and integrable random variables. Let us see that $\alpha$ can be extended into a norm on $\R^d$, using the following two relations: for any $x,y \in \Z^d$
\begin{equation}
\label{trianglealpha}
\alpha(x+y) \le \alpha(x) + \alpha(y),
\end{equation}
(this being a consequence of (\ref{triangleam})), and for any $x \in \Z^d$ and any $q \in \N$,
\begin{equation}
\label{trianglealpha2}
\alpha(q x) = q \alpha(x).
\end{equation}
Indeed, the first identity implies that there exists $C > 0$ such that for any $y \in \Z^d$,
\begin{equation}
\label{continuitymodulus}
0 \le \alpha(y) \le C \|y\|_1.
\end{equation}
Moreover, invariance of the distribution of the potentials with respect to the spatial transformation $z \to -z$ ensures that $\alpha(x) = \alpha(-x)$, so we derive from (\ref{trianglealpha}) and (\ref{continuitymodulus}) that
\begin{equation}
\label{continuityalpha}
|\alpha(x+y) - \alpha(x)| \le C \|y \|_1.
\end{equation}
Identity (\ref{trianglealpha2}) gives us a way to extend $\alpha$ to all rationals. This extension preserves (\ref{continuityalpha}) by homogeneity, so one can finally extend $\alpha$ to $\R^d$ by continuity. The only missing element to establish that $\alpha$ is indeed a norm is to check that $\alpha(x) \neq 0$ if $x \neq 0$. Jensen's inequality ensures that 
\begin{eqnarray*}
\label{applyjensen}
\E[a_m(0,x)] & \ge & - \log \E \EE_{\m(0)}\left[ \exp\left( -\sum_{n = 0}^{H_{\m(x)} - 1}  V(S_n) \rr) \ \1_{\{H_{\m(x)} < +\infty\}} \right] \\
& \ge & -\log \EE_{\m(0)} \E \left[ \exp\left( -\sum_{y :  H_y < H_{\m(x)}}  V(y) \rr) \ \1_{\{H_{\m(x)} < +\infty\}} \right] \\
& \ge & -\log \EE_{\m(0)} \left[ \prod_{y :  H_y < H_{\m(x)}} \E\ll[ e^{-V(y)} \rr] \ \1_{\{H_{\m(x)} < +\infty\}} \right] .
\end{eqnarray*}
As any path leading from $\m(0)$ to $\m(x)$ visits at least $\|x\|_1 - 2$ vertices, we obtain 
$$
\E[a_m(0,x)] \ge -(\|x\|_1 - 2) \log \E \ll[ e^{-V(0)} \rr].
$$
Together with the $L^1$ convergence in (\ref{limam}) to $\alpha(x)$, this yields that 
$$
\alpha(x) \ge - \|x\|_1 \log \E \ll[ e^{-V(0)} \rr],
$$
which is strictly positive for any $x \neq 0$ since we assume the potentials to be non identically zero. 
\end{proof}

%
%
%
%
%
%
%
\section{Shape theorem under a moment condition}
\label{s:shapemoment}
\setcounter{equation}{0}
In this section, we will prove part (1) of Theorem~\ref{t:shape}. First, let us see that $\E[Z(0)^d] < + \infty$ is a necessary condition for (\ref{eq:shapemoment}) to hold. Recall that we write $D(x,r)$ for the $\|\cdot\|_1$-ball of centre $x$ and radius $r$.

As $\alpha$ is a norm, the set $K$ defined in (\ref{defK}), which is the unit ball for $\alpha$, contains a neighbourhood of the origin. In order for (\ref{eq:shapemoment}) to be true, it must be that, for some $\delta > 0$, one has
\begin{equation}
\label{interiorAt}
D(0,\delta) \subset t^{-1} A_t^\circ \text{ for all } t \text{ large enough}.
\end{equation}
If $Z(0)^d$ is not integrable, then for any $M > 0$, one has
$$
\sum_{x \in \Z^d} \P[Z(x) > M \|x\|_1] = + \infty.
$$
As $a(0,x) \ge Z(x)$ for $x \neq 0$, it follows from this observation that almost surely, there is an infinite number of $x$'s in $\Z^d$ such that $a(0,x) \ge M \|x\|_1$. This is in contradiction with (\ref{interiorAt}) as soon as $M > 1/\delta$.

We will now show that the condition $\E[Z(0)^d] < + \infty$ is sufficient for (\ref{eq:shapemoment}) to hold. To begin with, we extend the definition of $a$: for $x, y \in \R^d$, we set $a(x,y) = a(\lfloor x \rfloor,\lfloor y \rfloor)$, where $\lfloor x \rfloor$ is obtained from $x$ by applying the floor function on each coordinate. Note that
$$
\{ x \in \R^d : a(0,x) \le t \} = A_t^\circ,
$$
where $A_t^\circ$ was defined in (\ref{defAtcirc}). Similarly, we set $a_m(x,y) = a_m(\lfloor x \rfloor,\lfloor y \rfloor)$. The almost sure existence of radial limits, as stated in part (2) of Theorem~\ref{t:existpoint}, extends to points with rational coordinates.
\begin{prop}
\label{radiallimits}
Assume that $Z(0)$ is integrable. For any $x \in \Q^d$, one has
\begin{equation}
\label{radiala}
\frac1n a_m(0,nx) \xrightarrow[n \to +\infty]{\text{a.s.}} \alpha(x),
\end{equation}
and the same is true if $a_m$ is replaced by $a$.
\end{prop}
\begin{proof}
To see that the statement of the proposition is true for $a_m$, one can proceed as in \cite[p.~592]{cd} (that is, through a Borel-Cantelli argument relying on integrability of $a_m$). In order to obtain the result for $a$, one can use the inequalities in (\ref{triangleam}), and reason as in the proof of part (2) of Theorem~\ref{t:existpoint}. 
\end{proof}
Let 
$$
A_{m,t}^\circ = \{ x \in \R^d : a_m(0,x) \le t \}.
$$
The central ingredient of the proof of part (1) of Theorem~\ref{t:shape} is the following control of the tail probability of $a_m(0,x)$, which parallels \cite[Lemma~3.3]{cd} and \cite[Lemma~7]{zer}.
\begin{prop}
\label{controltail}
Assume that $\E[Z(0)^d] < + \infty$. There exist a constant $c_1 > 0$ and a random variable $X$ with finite $d$-th moment such that, for any $x \in \R^d$ and any $t > 0$,
$$
\P[a_m(0,x) > t] \le c_1 \frac{\|x\|_1^{2d}}{(t-c_1 \|x\|_1)_+^{4d}} +  \P[X > t],
$$
where for $s \in \R$, we write $s_+ = \max(0,s)$.
\end{prop}
\begin{proof}
Let $L$ be a positive integer, and $\td{L} = 2L+1$. Note that $(B(\td{L}z,L))_{z \in \Z^d}$ is a partition of $\Z^d$. For $y \sim z \in \Z^d$, let us write $T^L_{y,z}$ for the exit time from the ``rectangle'' $R^L(y,z) = B(\td{L}y,L) \cup B(\td{L}z,L)$. For such $y$ and $z$, we define 
$$
a_m^L(\td{L}y,\td{L}z) = - \log \EE_{\m(\td{L}y)}\left[ \exp\left( -\sum_{n = 0}^{H_{\m(\td{L}z)} - 1}  V(S_n) \rr) \ \1_{\{H_{\m(\td{L}z)} < T^L_{y,z}\}} \right].
$$	
This is clearly an upper bound for $a_m(\td{L}y,\td{L}z)$. We want to check that for $L$ large enough, $a_m^L$ has finite $d$-th moment, by adapting the proof of Proposition~\ref{integram}. In order that (\ref{amxyle}) still holds for $a_m^L(Ly,Lz)$ instead of $a_m(x,y)$, it suffices to have $L$ large enough so that the paths $\gamma^1,\ldots, \gamma^{2d}$ that appear in (\ref{amxyle}) are all contained in the rectangle $R^L(y,z)$. This is possible, since we learn from Lemma~\ref{2dpaths} that the paths can be chosen so that the length of the longest one does not exceed $\|\m(\td{L}z) - \m(\td{L}y)\|_1 + 8$ (in fact, choosing $L = 5$ is sufficient). For such $L$, we then see that $a_m^L(Ly,Lz)$ has finite $d$-th moment by applying Lemma~\ref{integrpath}. 

Let $\gamma^1,\ldots,\gamma^{2d}$ be the $2d$ paths obtained from Lemma~\ref{2dpaths} that connect $0$ to $\lfloor x/\td{L} \rfloor$. We write $\gamma^i = (\gamma^i_0,\ldots, \gamma^i_{l^i})$, and assume that $x$ is far enough from $0$ so that for any $i$, $l^i \ge 2$. Applying (\ref{triangleam}) iteratively, we have
\begin{equation}
\label{decompamamL}
a_m(0,x) \le a_m(0,\td{L}\gamma^i_1) + \sum_{k = 1}^{l_i-2} a_m^L(\td{L} \gamma^i_k,\td{L} \gamma^i_{k+1}) + a_m(\td{L} \gamma^i_{l_i-1},x).
\end{equation}
We rewrite the r.h.s.\ of (\ref{decompamamL}) as $A^i_1 + A^i_2 + A^i_3$, with obvious identifications. Note that in fact,
\begin{equation}
\label{decompamA'}
a_m(0,x) \le \min_{1 \le i \le 2d} \ll( A^i_1 + A^i_2 + A^i_3 \rr)  \le \sum_{i = 1}^{2d} A^i_1 + \min_{1 \le i \le 2d} A^i_2 + \sum_{i = 1}^{2d} A^i_3,
\end{equation}
and again we rewrite the r.h.s.\ of (\ref{decompamA'}) as $A'_1 + A'_2 + A'_3$. One has
$$
\P[a_m(0,x) > 2 t] \le \P[A'_2 > t] + \P[A'_1 + A'_3 > t].
$$
The random variable $A'_1 + A'_3$ has finite $d$-th moment, so it is tempting to take this as the $X$ appearing in the proposition, but we need instead to find a random variable that does not depend on $x$. Let us see for instance that $a_m(\td{L} \gamma^i_{l_i-1},x)$ is stochastically dominated by a random variable with finite $d$-th moment, uniformly over $x$. An examination of the proof of Lemma~\ref{integrpath} shows that for this to be true, it suffices to find $2d$ paths connecting $\m(\td{L}\gamma^i_{l_i-1})$ to $\m(\lfloor x \rfloor)$ whose lengths are bounded uniformly over $x$. But this is indeed possible, as we can always take paths of length smaller than $\|\m(\td{L}\gamma^i_{l_i-1})-\m(\lfloor x \rfloor)\|_1 +8 \le 2\td{L}d+8$.

Let us now turn to the estimation of $\P[A'_2 > t] = \prod_{i = 1}^{2d} \P[A^i_2 > t]$. The construction of $a_m^L$ was made in such a way that if $\{y,z\}$ and $\{y',z'\}$ are disjoint pairs of neighbouring sites, then $a_m^L(\td{L}y,\td{L}z)$ and $a_m^L(\td{L}y',\td{L}z')$ are independent random variables. Moreover, it is clear for the same reason as above that one can find a constant $C$ such that $\E[(a_m^L(\td{L}y,\td{L}z))^2] \le C$ uniformly over all $y,z \in \Z^d$ with $y \sim z$. Hence,
$$
\var\ll[A^i_2 \rr] \le 3 l_i C \le 3 (\|x\|_1 + 8) C,
$$
where $\var$ denotes the variance with respect to $\P$. Naturally, the first moment of $a_m^L(\td{L}y,\td{L}z)$ is also uniformly bounded over all $y,z \in \Z^d$ with $y \sim z$, so possibly enlarging $C$, we have
$$
\E[A^i_2] \le C l_i \le C (\|x\|_1 + 8).
$$
As a consequence,
$$
\P[A^i_2 > t + C(\|x\|_1 + 8)] \le \frac{3 (\|x\|_1 + 8) C}{t^2},
$$
which leads to the claim of the proposition.
\end{proof}
For $x \in \R^d$ and $\delta > 0$, we define 
$$
\mcl{D}_{\delta,x} = \Z^d \cap \bigcup_{t \ge 0} D(t x,\delta t),
$$
and if $y \in \mcl{D}_{\delta,x}$, we let
$$
\sigma_{\delta,y} = \inf \{t \ge 0 : y \in D(t x,\delta t) \}.
$$
Since the function $t \mapsto \|y-t x\|_1 - \delta t$ is continuous, either the infimum above is $0$, or 
\begin{equation}
\label{equalsigma}
\|y - \sigma_{\delta,y} x\|_1 =\delta  \sigma_{\delta,y}.
\end{equation}
If the infimum is $0$, then it must be that $y = 0$, and in this case \eqref{equalsigma} also holds.
\begin{prop}
\label{coverball}
Assume that $\E[Z(0)^d] < + \infty$. 
\begin{enumerate}
\item For any $\eps > 0$, there exists $\delta_1 > 0$ such that, for any $\delta \le \delta_1$, any $x \in \Q^d$ with $\alpha(x) \le 1 - 3 \eps$, for all but a finite number of $y \in \mcl{D}_{\delta,x}$, one has
\begin{equation}
\label{amepssigma}
a_m(0,y) \le (1- \eps) \sigma_{\delta,y}.
\end{equation}
\item For any $\eps > 0$, one has
\begin{equation}
\label{susbsetAmt}
\P[(1-\eps)K \subset t^{-1} A_{m,t}^\circ \text{ for all } t \text{ large enough}] = 1,
\end{equation}
where $K$ was defined in (\ref{defK}).
\item For any $\eps > 0$, (\ref{susbsetAmt}) holds with $A_{m,t}^\circ$ replaced by $A_{t}^\circ$.
\end{enumerate}
\end{prop}
\begin{proof}
We follow the line of argument given in \cite[p.~594-595]{cd}. Let $\eps > 0$, $x \in \Q^d$ such that $\alpha(x) \le 1 - 3 \eps$, and $\delta \le \eps/(2c_1)$, where $c_1$ is the constant appearing in Proposition~\ref{controltail}. By Proposition~\ref{radiallimits}, we already know that a.s.,
\begin{equation}
\label{allbutafew}
\begin{array}{l}
\text{for all but a finite number of } y \in \mcl{D}_{\delta,x}, \\
a_m(0,\sigma_{\delta,y} x) \le (\alpha(x) + \eps) \sigma_{\delta,y}.
\end{array}
\end{equation}
We need to control the size of $a_m(\sigma_{\delta,y} x,y)$. Note that $\|y - \sigma_{\delta,y} x \|_1 = \delta \sigma_{\delta,y}$. With our particular choice of $\delta$, equation~\eqref{equalsigma} and Proposition~\ref{controltail} imply that
\begin{equation}
\label{rhssums}
\P[a_m(\sigma_{\delta,y} x,y) > \eps \sigma_{\delta,y}] \le c_1 \frac{(\delta \sigma_{\delta,y})^{2d}}{(c_1 \delta \sigma_{\delta,y})^{4d}} + \P[X > \eps \sigma_{\delta,y}].
\end{equation}
Now, using a triangle inequality on \eqref{equalsigma}, we learn that
\begin{equation}
\label{sigmage}
\sigma_{\delta,y} \ge \frac{\|y\|_1}{\delta + \|x\|_1},	
\end{equation}
so up to a constant, the sum over all $y \in \Z^d$ of the r.h.s.\ of (\ref{rhssums}) is bounded by
$$
\sum_{y \in \Z^d} \ll( \frac{\|y\|_1}{\delta + \|x\|_1} \rr)^{-2d} + \P\ll[ X > \frac{\eps \|y\|_1}{\delta + \|x\|_1} \rr].
$$
This sum is finite, which thus implies that, for all but a finite number of $y \in \mcl{D}_{\delta,x}$, one has
$$
a_m(\sigma_{\delta,y} x,y) \le \eps \sigma_{\delta,y}.
$$
Together with (\ref{allbutafew}) and inequality (\ref{triangleam}), this implies that for all but a finite number of $y \in \mcl{D}_{\delta,x}$, one has
$$
a_m(0,y) \le (\alpha(x) + 2\eps) \sigma_{\delta,y} \le (1-\eps) \sigma_{\delta,y},
$$
and thus part (1) of the proposition is proved.

We now turn to part (2). First, let us see that if $\delta_2$ is small enough, then 
\begin{equation}
\label{localcover0}
\P[D(0,\delta_2) \subset t^{-1} A_{m,t}^\circ \text{ for all } t \text{ large enough}] = 1.	
\end{equation}
Were this not true, with non-zero probability, there would exist two sequences $y_n \in \R^d$, $t_n \to + \infty$ such that
$$
a_m(0,y_n) \ge t_n \quad \text{ and } \quad \|y_n\|_1 \le \delta_2 t_n.
$$
In particular, on this event, we would have $\|y_n\|_1 \to +\infty$ and
\begin{equation}
\label{eq:localcov0}
a_m(0,\lfloor y_n \rfloor) = a_m(0,y_n) \ge \frac{\|y_n\|_1}{\delta_2} \ge \frac{\|\lfloor y_n \rfloor\|_1 - d}{\delta_2}.
\end{equation}
But using the tail estimate on $a_m(0,y)$ from Proposition~\ref{controltail}, we can ascertain that with probability one, for all but a finite number of $y \in \Z^d$, one has
\begin{equation}
\label{2c1}
a_m(0,y) \le 2c_1 \|y\|_1.	
\end{equation}
This would contradict (\ref{eq:localcov0}) as soon as $\delta_2 < 1/(2 c_1)$, so identity (\ref{localcover0}) holds for such $\delta_2$. 

Let $\eps > 0$. We claim that, with this $\delta_2$ now fixed, we can find some $\delta > 0$ small enough such that
\begin{equation}
\label{localcoverrest}
\begin{array}{l}
\text{if } x \in \Q^d \text{ is such that } \|x\|_1 \ge \delta_2 \text{ and } \alpha(x) \le 1-3\eps, \text{ then} \\
\displaystyle{ \P[D(x,\delta) \subset t^{-1} A_{m,t}^\circ \text{ for all } t \text{ large enough}] = 1.	   }
\end{array}
\end{equation}
Indeed, if this is not true for some $x$, then with non-zero probability, one can construct sequences $y_n \in \R^d$, $t_n \to +\infty$ such that
\begin{equation}
\label{amxn}
a_m(0,y_n) \ge t_n \quad \text{ and } \quad \|y_n - t_n x \|_1 \le \delta t_n.
\end{equation}
A triangle inequality on the second part of (\ref{amxn}) yields that
\begin{equation}
\label{obs1}
t_n \ge \frac{\|y_n\|_1}{\delta + \|x\|_1},
\end{equation}
while the ``converse'' triangle inequality applied to \eqref{equalsigma} tells us that
$$
\sigma_{\delta,\lfloor y_n \rfloor} \le \frac{\|\lfloor y_n \rfloor\|_1}{\|x\|_1 - \delta}\le \frac{\| y_n \|_1+d}{\|x\|_1 - \delta},
$$
provided $\delta < \|x\|_1$. We can rewrite the latter inequality as
\begin{equation}
\label{obs2}
\| y_n \|_1 \ge (\|x\|_1 - \delta)\sigma_{\delta,\lfloor y_n \rfloor} - d.
\end{equation}
Observations \eqref{obs1} and \eqref{obs2} lead us to
$$
a_m(0, \lfloor y_n \rfloor) = a_m(0,y_n) \ge \frac{\|y_n\|_1}{\delta + \|x\|_1} \ge \sigma_{\delta,\lfloor y_n \rfloor}  \frac{\|x\|_1 - \delta}{\delta + \|x\|_1} - \frac{d}{\delta}.
$$
Note also that
\begin{equation}
\label{whichdelta}
\frac{\|x\|_1 - \delta}{\delta + \|x\|_1} = 1 - \frac{2\delta}{\|x\|_1 + \delta} \ge 1 - \frac{2\delta}{\delta_2 + \delta}.	
\end{equation}
Let $\delta_1$ be given by part (1) of the proposition. We choose $\delta \le \delta_1$ small enough such that the r.h.s.\ of (\ref{whichdelta}) is larger than $1-\eps/2$. With this particular choice of $\delta$, we have shown that, if (\ref{localcoverrest}) fails to hold, then for some $x \in \Q^d$ with $\alpha(x) \le 1 - 3 \eps$, with non-zero probability there exists a sequence $y_n \in \R^d$ with $\|y_n\|_1 \to \infty$ and such that
$$
a_m(0, \lfloor y_n \rfloor) \ge \ll(1-\frac{\eps}{2}\rr) \sigma_{\delta,\lfloor y_n \rfloor} - \frac{d}{\delta}.
$$
Since $\lim_{\|y\| \to \infty} \sigma_{\delta,y}  = + \infty$ (see \eqref{sigmage}), this contradicts part (1) of the proposition, and thus (\ref{localcoverrest}) holds. To conclude, we note that it is possible to cover $(1-3\eps)K$ by the union of $D(0,\delta_2)$ and a finite numbers of balls each having radius $\delta$ and some centre $x$ as in (\ref{localcoverrest}), so part (2) of the proposition follows from (\ref{localcover0}) and (\ref{localcoverrest}).

Let us now turn to part (3). From Proposition~\ref{approxam}, we learn that
\begin{equation*}
a(0,y) \le a_m(0,y) + V(0) + Z(y) + 2 \log(2d).	
\end{equation*}
Since $Z(y)$ are identically distributed random variables with finite $d$-th moment, using inequality (\ref{sigmage}), it comes that whatever the value of $\delta > 0$, for all but a finite number of $y \in \Z^d$, one has
$$
V(0) + Z(y) + 2 \log(2d) \le \min\ll( \frac{\eps}{4} \sigma_{\delta,y} , \|y\|_1 \rr).
$$
As a consequence, part (1) of the proposition also holds if (\ref{amepssigma}) is replaced by 
$$
a(0,y) \le \ll (1- \frac{3 \eps}{4} \rr) \sigma_{\delta,y},
$$
and similarly, we get from (\ref{2c1}) that, for al but a finite number of $y \in \Z^d$, one has
$$
a(0,y) \le (2 c_1 + 1) \|y\|_1.
$$
These are sufficient ingredients to make the proof of part (2) of the proposition carry over to $A_t^\circ$, so the proof is complete.
\end{proof}

\begin{prop}
\label{shape2borne}
Assume that $\E[Z(0)^d] < + \infty$. 
\begin{enumerate}
\item For any $\eps > 0$, there exists $\delta_3 > 0$ such that, for any $\delta \le \delta_3$, any $x \in \Q^d$ with $\alpha(x) \ge 1 + 3 \eps$, for all but a finite number of $y \in \mathcal{D}_{\delta,x}$, one has
$$
a_m(0,y) \ge (1+\eps) \sigma_{\delta,y}.
$$
\item For any $\eps > 0$, one has
\begin{equation}
\label{eq:shape2}
\P[t^{-1} A_{m,t}^\circ \subset (1+\eps)K \text{ for all } t \text{ large enough}] = 1.
\end{equation}
\item For any $\eps > 0$, (\ref{eq:shape2}) holds if $A_{m,t}^\circ$ is replaced by $A_t^\circ$.
\end{enumerate}
\end{prop}
\begin{proof}
Let $\eps > 0$, $x \in \Q^d$ such that $\alpha(x) \ge 1 + 3 \eps$, and let $\delta \le \eps/(2c_1)$, where $c_1$ is the constant appearing in Proposition~\ref{controltail}. Using similar arguments as in part (1) of Proposition~\ref{coverball}, one can show that for all but a finite number of $y \in \mcl{D}_{\delta,x}$, 
\begin{equation}
\label{allbutafew2}
a_m(0,\sigma_{\delta,y} x) \ge (\alpha(x) - \eps) \sigma_{\delta,y} \quad \text{ and } \quad a_m(y,\sigma_{\delta,y} x) \le \eps \sigma_{\delta,y}.
\end{equation}
Inequality (\ref{triangleam}) ensures that
$$
a_m(0,y) \ge a_m(0,\sigma_{\delta,y} x) - a_m(y,\sigma_{\delta,y} x),
$$
so except for a finite number of $y \in \mcl{D}_{\delta,x}$, one has
$$
a_m(0,y) \ge (\alpha(x) - 2\eps) \sigma_{\delta,y} \ge (1+\eps)\sigma_{\delta,y},
$$
and part (1) is proved. 

Let us now turn to part (2), and see that one can find $\delta > 0$ small enough such that
\begin{equation}
\label{localcoverrest2}
\begin{array}{l}
\text{if } x \in \Q^d \text{ satisfies } \alpha(x) \ge 1+3\eps, \text{ then} \\
\displaystyle{ \P[D(x,\delta) \cap t^{-1} A^\circ_{m,t} = \emptyset \text{ for all } t \text{ large enough}] = 1.	   }
\end{array}
\end{equation}
The proof is very similar to the proof of (\ref{localcoverrest}), so we do not repeat it here (note that as $\alpha$ is a norm, the fact that $\alpha(x) \ge 1$ gives a uniform lower bound on $\|x\|_1$). 

We can construct a covering of the set 
$$
\{x \in \R^d : (1+3\eps) \le \alpha(x) \le 2 \}
$$
by a finite union of balls each having radius $\delta$ and centre some $x \in \Q^d$ with $\alpha(x) \ge 1+ 3\eps$. We thus obtain from (\ref{localcoverrest2}) that, for all sufficiently large $t$,
$$
t^{-1} A^\circ_{m,t} \cap \{x \in \R^d : (1+3\eps) \le \alpha(x) \le 2 \} = \emptyset,
$$
but this can only happen if (\ref{eq:shape2}) is true.

Part (3) is obtained in the same way as part (3) of Proposition~\ref{coverball}, using 
$$
 a(0,y) \ge a_m(0,y) - Z(0) - 2 \log(2d) - u_m(y)
$$
from Proposition~\ref{approxam}, together with the fact that $(u_m(y))_{y \in \Z^d}$ are identically distributed random variables with finite $d$-th moment (the fact that they are identically distributed comes from our construction of the paths $\gamma^i_{x,y}$ in (\ref{defgammaxy}), and finiteness of the $d$-th moment comes from Proposition~\ref{integram}).
\end{proof}

%
%
%
%
%
%
%
\section{Some discrete geometry}
\label{s:geom}
\setcounter{equation}{0}

In order to use duality arguments in high dimensional percolation, we need to study the relationship between sets and their boundaries, and consider when (and in which sense) the latter form connected sets. Given a set $A \subset \Z^d$, we define its \emph{outer} and \emph{inner boundaries} by, respectively,
$$
\ov{\partial} A = \{ x \in \Z^d \setminus A : \exists y \in A, x \sim y \},
$$
$$
\underline{\partial} A = \{ x \in A : \exists y \in \Z^d \setminus A, x \sim y \}.
$$
For any two points $x,y \in \Z^d$, we say that they are $*$-neighbours, and write $x \stackrel{*}{\sim} y$, if $\|y-x\|_\infty = 1$. We say that a set $A \subset \Z^d$ is $*$-connected if it is connected for this new adjacency relation. We define the $*$-\emph{outer} and $*$-\emph{inner boundaries}
of $A$ by, respectively,
$$
\ov{\partial}^* A = \{ x \in \Z^d \setminus A : \exists y \in A, x \stackrel{*}{\sim} y \},
$$
$$
\underline{\partial}^* A = \{ x \in A : \exists y \in \Z^d \setminus A, x \stackrel{*}{\sim} y \}.
$$
If we want to emphasize that we consider the usual nearest-neighbour relation $\sim$ on $\Z^d$, we may talk about $\Z^d$-neighbours, $\Z^d$-connected sets, and so on.

We say that a set $A \subset \Z^d$ \emph{has a} $\Z^d$-\emph{hole} if one of the $\Z^d$-connected components of $\Z^d \setminus A$ is finite.
\begin{prop}
\label{p:geom}
If $A \subset \Z^d$ is a finite $*$-connected set without $\Z^d$-hole, then $\underline{\partial} A$ is $*$-connected, and $\ov{\partial}^* A$ is $\Z^d$-connected.
\end{prop}
The first part of the Proposition is taken from \cite[Lemma~2.1 (i)]{deupis}, and the second part from \cite[Lemma~2.23]{kes} (although one can convince oneself that the proposition should be true after a few drawings, as illustrated in figure~\ref{fig:neighbours}, it turns out that the proofs of these facts are quite involved). 

\begin{figure}
\centering
\includegraphics[scale=0.5]{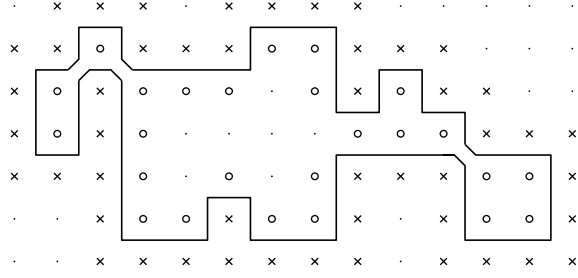}
\caption{
\small{
A $*$-connected set without $\Z^d$-hole, delimited by the closed contour. Crosses denote the ($\Z^d$-connected) $*$-outer boundary of this set, and circles its ($*$-connected) $\Z^d$-inner boundary.
}
}
\label{fig:neighbours}
\end{figure}

%
%
%
%
%
%
%
%
\section{The renormalization construction}
\label{s:renormalize}
\setcounter{equation}{0}

We fix once and for all some $M > 0$ such that 
\begin{equation}
\label{defM}
\P[V(0) \le M] > p_c .	
\end{equation}
A site $x \in \Z^d$ is said to be \emph{livable} if $V(x) < + \infty$, and to be \emph{healthy} if $V(x) \le M$. We call a connected component of livable sites (resp.\ healthy sites) a \emph{livable cluster} (resp.\ \emph{healthy cluster}). We define similarly \emph{livable paths} (resp.\ \emph{healthy paths}) as paths visiting only livable sites (resp.\ healthy sites).

Let $N \ge 2$ be an even integer. In order to lighten the notation and stress the fact that once suitably chosen, $N$ will be kept fixed throughout, we write $B_i = B((2N+1)i,N)$ and $B'_i = B((2N+1)i,3N/2)$, for any $i \in \Z^d$ (recall the definition of $B(\cdot,\cdot)$ given in (\ref{defBxr})).

The boxes $(B_i)_{i \in \Z^d}$ form a partition of $\Z^d$. For any $x \in \Z^d$, we let $i(x)$ be the index such that $x \in B_{i(x)}$. We may refer to the indices of the boxes as the \emph{macroscopic scale}, and to elements of the boxes as the \emph{microscopic scale}.

We say that $i \in \Z^d$ is \emph{good} if the following two conditions hold:
\begin{enumerate}
\item $B_i'$ contains a \emph{crossing cluster} $\texttt{CC}_i$, that is, a healthy cluster linking any two opposite faces of the box,
\item In $B_i' \setminus \texttt{CC}_i$, there is no livable cluster of diameter larger than $N/4$. 
\end{enumerate}
We say that $i$ is \emph{bad} otherwise.

\begin{prop}
\label{p:lss}
For any $p < 1$, there exists $N$ such that $(\1_{\{i \text{ is good}\}})_{i \in \Z^d}$ stochastically dominates an independent Bernoulli percolation of parameter $p$.
\end{prop}
\begin{proof}
Since the family of random variables $(\1_{\{i \text{ is good}\}})_{i \in \Z^d}$ have only finite-range dependence, and considering \cite[Theorem~B26 p.\ 14]{lig} (which is a special case of results in \cite{lss}), it suffices to verify that the probability for $0$ to be good can be made as close to $1$ as desired.

The probability for a site of the microscopic scale to be healthy is supercritical. Hence, the probability that there exists a connected component of healthy sites linking the two opposite faces of $B_0'$
in the direction of the first coordinate axis tends to $1$ as $N$ tends to infinity (see \cite[Theorem~8.97]{grim}). The same is true in the other coordinate directions, so the probability of existence of a crossing cluster in $B_0'$ tends to $1$ as $N$ tends to infinity. 

We now argue that the probability that there exist two disjoint livable clusters in $B_0'$ with diameters larger than $N/4$ tends to $0$. Indeed, on this event, due to the uniqueness of the infinite livable cluster, there exists at least one livable cluster which is both finite and of diameter larger than $N/4$ inside $B_0'$. The probability of this event is bounded by
$$
\sum_{x \in B_0'} \P[N/4 \le \mathrm{diam}(\text{livable cluster containing } x) < + \infty],
$$
which decays to $0$ exponentially fast due to \cite[Theorem~8.21]{grim}.
\end{proof}
\begin{equation}
\label{t:lss}
\begin{array}{l}
\text{From now on, we fix } N \text{ such that the conclusion of} \\
\text{Proposition~\ref{p:lss} holds with } p > \max(1-3^{-d},p_c). 
\end{array}
\end{equation}

\begin{prop}
\label{uniqueinfinitecomp}
With probability one, there exists a unique infinite connected component of good macroscopic sites, that we denote by $\mclC_\infty$.
\end{prop}
\begin{proof}
Let $(\mfk{B}_i)_{i \in \Z^d}$ be independent Bernoulli random variables of parameter $p$ as in (\ref{t:lss}) that are dominated by $(\1_{\{i \text{ is good}\}})_{i \in \Z^d}$. The set $\{ i \in \Z^d : \mfk{B}_i = 1 \}$ contains a unique infinite connected component, that we denote by $\mfk{C}_\infty$. By domination, this set is contained in the set of good sites, hence there exists at least one infinite connected component of good sites.

If an infinite connected component of good sites does not contain $\mfk{C}_\infty$, then it belongs to its complement. In order to check uniqueness, it thus suffices to show that there exists no infinite connected component in the complement of $\mfk{C}_\infty$. By translation invariance, it suffices to show that the connected component of $\Z^d \setminus \mfk{C}_\infty$ containing $0$ is finite almost surely. Let us write $\mfk{C}_0$ for this set. It is the set of points that can be connected to $0$ via a nearest-neighbour path visiting only sites of $\Z^d \setminus \mfk{C}_\infty$ (and it is empty if $0 \in \mfk{C}_\infty$). Note that, since $\mfk{C}_\infty$ is connected, the set $\mfk{C}_0$ has no $\Z^d$-hole.

Let $\eps > 0$ and $n \in \N$. We let $\mfk{C}_{0,n}$ be the connected component of $\mfk{C}_0 \cap B(0,n)$ containing $0$. If $|\mfk{C}_0| = +\infty$, then there exists a path in $\mfk{C}_{0,n}$ linking $0$ to one face of the box $B(0,n)$. For definiteness, let us assume that it is the face $\{n\}\times \{-n,\ldots,n\}^{d-1}$. By \cite[Theorem~8.97]{grim}, we know that for all $n$ large enough, the probability that $B(0,n) \cap \mfk{C}_\infty$ contains an open path linking any two opposite faces of $B(0,n)$ becomes larger than $1-\eps$. On this event, any slice of the form $\{k\}\times \{-n,\ldots,n\}^{d-1}$ with $0 \le k \le n$ contains points of both $\mfk{C}_{0,n}$ and its complement, and therefore contains at least one point in $\underline{\partial}\mfk{C}_{0,n}$. We have thus proved that, outside of an event of probability less than $\eps$, the fact that $|\mfk{C}_0| = +\infty$ implies that there are at least $n$ points in $\underline{\partial}\mfk{C}_{0,n} \cap B(0,n-1)$.

Say that a site $i$ is \emph{red} if $\mfk{B}_i = 0$. Any point in $\underline{\partial}\mfk{C}_{0,n} \cap B(0,n-1)$ is red, since such a point must be outside of $\mfk{C}_\infty$, while neighbouring this set. We learn from part (i) of Proposition~\ref{p:geom} that $\underline{\partial}\mfk{C}_{0,n}$ must be $*$-connected. We thus obtain
\begin{eqnarray*}
\P[|\mfk{C}_0| = +\infty] & \le & \eps + \sum_{x \in B(0,n)} \P
\left[
\begin{array}{l}
\exists *\text{-n.n.\ simple path of red sites} \\
\text{starting from } x \text{ and of length } n
\end{array}
\right] \\
& \le & \eps + \sum_{x \in B(0,n)} (1-p)^n (3^d)^n,
\end{eqnarray*}
with $p > 1-3^{-d}$. Taking $n$ large enough thus ensures that $\P[|\mfk{C}_0| = +\infty] \le 2\eps$, for any $\eps > 0$, and the result follows.
\end{proof}

\begin{figure}
\centering
\psfrag{aa}{\small{$i$}}
\includegraphics[scale=0.8]{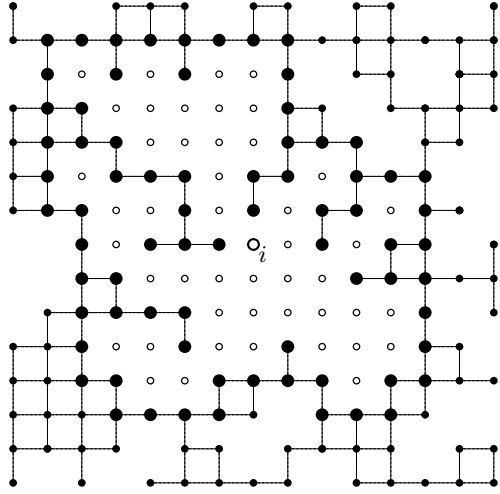}
\caption{
\small{
Macroscopic sites around $i \in \Z^d$. Black dots are the points of $\mcl{C}_\infty$, white dots the points of $\mcl{C}_i$, and dashed links are drawn between neighbouring black dots for clarity. Note that $\mcl{C}_i$ has no $\Z^d$-hole, while $\ov{\partial}^* \mathcal{C}_i$ is a $\Z^d$-connected subset of $\mcl{C}_\infty$, denoted by larger black dots.
}
}
\label{fig:mclC}
\end{figure}

For $i \notin \mclC_\infty$, let $\mathcal{C}_i$ be the smallest $*$-connected set containing $i$ such that any point of its $*$-outer boundary is in $\mclC_\infty$ (see figure~\ref{fig:mclC}). It is the set of points that can be connected to $i$ through a $*$-nearest-neighbour path visiting only points of $\Z^d \setminus \mclC_\infty$. Since $\mclC_\infty$ is $\Z^d$-connected, $\mclC_i$ has no $\Z^d$-hole. We let $\ovmclC_i = \mathcal{C}_i \cup \ov{\partial}^* \mathcal{C}_i$. For $i \in \mclC_\infty$, we let $\mclC_i = \emptyset$, and with some abuse of notation, $\ovmclC_i = \ov{\partial}^* \mathcal{C}_i = \{i\}$.  The next proposition guarantees that the set $\ovmclC_i$ cannot be very large.
\begin{prop}
\label{tailradius}
There exists $c > 0$ such that for any $i \in \Z^d$ and any $t \ge 0$,
$$
\P[|\ovmclC_i| \ge t] \le e^{-ct^{1-1/d}}.
$$
\end{prop}
\begin{proof}
By translation invariance, we can restrict our attention to $i=0$, and as $|\ovmclC_0| \le 3^d |\mclC_0| + 1$, it suffices to prove the proposition with $\ovmclC_0$ replaced by $\mclC_0$. 

To see that $\P[|\mclC_0| = +\infty] = 0$, one can reproduce the proof of the ``uniqueness'' part of Proposition~\ref{uniqueinfinitecomp}. Indeed, let $\eps > 0$, $n \in \N$ and let $\mclC_{0,n}$ be the connected component of $\mclC_0 \cap (0,n)$ containing $0$. The same argument (together with the stochastic domination of Proposition~\ref{p:lss} for the existence of crossings) ensures that outside of an event of probability less than $\eps$, there exists a $*$-connected set of $n$ bad sites in $\underline{\partial}\mclC_{0,n}$ whenever $|\mclC_0| = +\infty$. As a consequence,
\begin{eqnarray*}
\P[|\mclC_0| = +\infty] & \le & \eps + \sum_{x \in B(0,n)} \P
\left[
\begin{array}{l}
\exists *\text{-n.n.\ simple path of bad sites} \\
\text{starting from } x \text{ and of length } n
\end{array}
\right] \\
& \le & \eps + \sum_{x \in B(0,n)} (1-p)^n (3^d)^n,
\end{eqnarray*}
with $p > 1-3^{-d}$, and we thus obtain that indeed $\P[|\mclC_0| = +\infty] = 0$.

Let us now recall the isoperimetric inequality on $\Z^d$ (see for instance \cite[Section~I.4.B]{woess}): there exists a constant $\mcl{I}$ such that for any finite $A \subset \Z^d$, 
\begin{equation}
\label{isop}
|\underline{\partial}{A}|^{d/(d-1)} \ge \mcl{I} |A|.
\end{equation}
As a consequence,
\begin{equation}
\label{probisop}
\P\left[|\mclC_0| = n\right] \le \P\left[|\underline{\partial}\mclC_0| \ge (\mcl{I}n)^{1-1/d}\right].
\end{equation}
As before, the set $\underline{\partial}\mclC_0$ contains only bad sites and is $*$-connected. Moreover, on the event $|\mclC_0| = n$, the inner boundary of $\mclC_0$ must be contained in $B(0,n)$. The probability in (\ref{probisop}) is thus bounded by the probability to see a $*$-connected set of bad sites of size at least $(\mcl{I}n)^{1-1/d}$ in the box $B(0,n)$. Using Proposition~\ref{p:lss}, we arrive at
\begin{eqnarray*}
&& \P\left[|\mclC_0| = n\right] \\
&& \qquad \le \sum_{x \in B(0,n)} \P[\exists *\text{-n.n.\  simple path starting from } x \text{ and of length } (\mcl{I}n)^{1-1/d}] \\
&& \qquad \le \sum_{x \in B(0,n)} (1-p)^{(\mcl{I}n)^{1-1/d}} (3^d)^{(\mcl{I}n)^{1-1/d}},
\end{eqnarray*}
with $1 - p < 3^{-d}$, and this proves the proposition.
\end{proof}
Let 
\begin{equation}
\label{defDelta}
\Delta_i' = \bigcup_{j \in \ovmclC_i} B_j', \qquad \Delta^g_i = \bigcup_{j \in \ov{\partial}^* \mathcal{C}_i} \CC_j
\end{equation}
(the superscript ``$g$'' standing for ``good''). For a microscopic site $x \in \Z^d$, we let 
$\Delta'(x) = \Delta'_{i(x)}$ and $\Delta^g(x) = \Delta^g_{i(x)}$. We first make a few simple observations on the geometry of these sets.
\begin{prop}
\label{observ} 
\begin{enumerate}
\item If $\mclC_i \cap \mclC_j \neq \emptyset$, then $\mclC_i = \mclC_j$.
\item If $i,j \in \Z^d$ are good sites and $i \sim j$, then $\CC_i \cap \CC_j \neq \emptyset$.
\item If $x,y \in \Z^d$ are neighbours, then either $i(x) = i(y)$, or $i(x) \sim i(y)$.
\item If $\gamma = (\gamma_0,\ldots,\gamma_l)$ is a nearest-neighbour path visiting only livable sites, such that $\gamma_0 = x$ and $\gamma_l \notin \Delta'(x)$, then $\gamma$ visits a point of $\Delta^g(x)$.
\item For any $x \in \Z^d$, $\Delta^g(x)$ is a $\Z^d$-connected set.
\end{enumerate}
\end{prop}
\begin{proof}
To prove part (1), let $k \in \mclC_i \cap \mclC_j$ and $l \in \mclC_i$. There exists a $*$-nearest-neighbour path visiting only sites of $\Z^d \setminus \mathcal{C}_\infty$ between any two sites chosen in one of the pairs $(i,k)$, $(j,k)$ and $(i,l)$. By a suitable concatenation, one can construct a similar path linking $j$ to $l$ through $k$ and $i$, and thus $l \in \mclC_j$. We proved that $\mclC_i \subset \mclC_j$, and by symmetry $\mclC_i = \mclC_j$.

For part (2), let us say for simplicity that $e = j -i$ is in the direction of the first coordinate axis, that is, $j-i = (1,0,\ldots,0)$. The crossing property implies that $\CC_i \cap B'_j$ is a healthy cluster linking the face 
$$
(2N+1)j + \{-3N/2\} \times \{-3N/2,\ldots, 3N/2\}^{d-1}
$$ 
to the face 
\begin{multline*}
(2N+1)i + \{3N/2\} \times \{-3N/2,\ldots, 3N/2\}^{d-1} \\
= (2N+1)j + \{-N/2-1\} \times \{-3N/2,\ldots, 3N/2\}^{d-1}.
\end{multline*}
In particular, $\CC_i \cap B'_j$ is a livable cluster contained in $B'_j$ of diameter at least $N-1$, and as $j$ is good, it must be that $\CC_i \cap B'_j$ is contained in $\CC_j$.

Part (3) being easily checked, we now turn to part (4). After removing repetitions, the sequence $(i(\gamma_k))_{k \le l}$ forms a nearest-neighbour path starting in $\mclC_{i(x)}$ and ending out of it. Hence, there exists an index $k \le l$ such that $i(\gamma_k)$ is in the outer boundary of $\mclC_{i(x)}$, and in particular such that $i(\gamma_k)$ is good. As $\gamma_l \notin \Delta'(x)$, the path, which is in the box $B_{i(\gamma_k)}$ at time $k$, must then exit $B'_{i(\gamma_k)}$ through a sequence of at least $N/2$ livable sites. As $i(\gamma_k)$ is a good site, this is possible only if the path intersects $\CC_{i(\gamma_k)}$, which is a subset of $\Delta^g(x)$.

As concerns part (5), note first that if $j,k \in \ov{\partial}^* \mathcal{C}_i$ are neighbours, part (2) of the proposition ensures that $\CC_j \cup \CC_k$ is a connected set. By induction, one obtains that if there exists a $\Z^d$-nearest-neighbour path $\gamma = (\gamma_0,\ldots,\gamma_l)$ staying inside $\ov{\partial}^* \mathcal{C}_i$, then $\cup_k \CC_{\gamma_k}$ forms a connected set. The claim then follows from Proposition~\ref{p:geom}, which ensures that $\ov{\partial}^* \mathcal{C}_i$ is a $\Z^d$-connected set, so one can find a $\Z^d$-nearest-neighbour path covering $\ov{\partial}^* \mathcal{C}_i$. 
\end{proof}

%
%
%
%
%
%
%
%
\section{Point-to-point exponent without moment condition}
\label{s:ptpnomoment}
\setcounter{equation}{0}
The aim of this section is to prove part (3) of Theorem~\ref{t:existpoint}.
Let 
$$
\tau_g(y) = \inf \{ n \in \N : S_n \in \Delta^g(y) \}.
$$
We define
\begin{equation}
\label{deftda}
\td{a}(x,y) = - \log \EE_{x} \left[ \exp\left( -\sum_{n = 0}^{\tau_g(y) - 1} V(S_n) \right), \tau_g(y) < + \infty \right],
\end{equation}
and 
\begin{equation}
\label{defhata}
\hat{a}(x,y) = \min_{x' \in \Delta^g(x)} \td{a}(x',y).
\end{equation}
The quantity $\hata(x,y)$ is an approximation of $a(x,y)$ in the following sense.
\begin{prop}
\label{compahata}
Let $x \in \Z^d$. If there exists a nearest-neighbour path of livable sites connecting $x$ to $\Delta^g(x)$, let $\gamma_x$ be one of minimal length, chosen according to some deterministic and translation-invariant rule in case of ties. In this case, we define
\begin{equation}
\label{defux}
u(x) = \sum_{z \in \gamma_x \cup \Delta^g(x)} (V(z)+\log(2d)),
\end{equation}
and set $u(x) = +\infty$ otherwise. For any $x,y \in \Z^d$, one has
\begin{equation}
\label{triangleineq}
\hata(x,y) \le a(x,y) \le \hata(x,y) + u(x) + u(y).	
\end{equation}
\end{prop}
\begin{proof}
To prove the first inequality in (\ref{triangleineq}), we distinguish three cases.

\noindent \emph{Case 1.} If $\ovmclC_{i(x)} \cap \ovmclC_{i(y)} \neq \emptyset$, then we claim that $\ov{\partial}^* \mclC_{i(x)} \cap \ov{\partial}^* \mclC_{i(y)} \neq \emptyset$. Indeed, it is a consequence of part (1) of Proposition~\ref{observ} if $\mclC_{i(x)} \cap \mclC_{i(y)} \neq \emptyset$. Otherwise, let $i \in \ov{\partial}^* \mclC_{i(x)} \cap \ovmclC_{i(y)}$. As $i \in \ov{\partial}^* \mclC_{i(x)}$, it must be that $i$ belongs to $\mclC_\infty$, and this implies that $i \in \ov{\partial}^* \mclC_{i(y)}$, so the claim is justified.

Let $i \in \ov{\partial}^* \mclC_{i(x)} \cap \ov{\partial}^* \mclC_{i(y)}$. Then both $\Delta^g(x)$ and $\Delta^g(y)$ contain $\CC_{i}$, and thus $\hata(x,y) = 0 \le a(x,y)$.

\noindent \emph{Case 2.} Assume now that $d(\ovmclC_{i(x)}, \ovmclC_{i(y)}) = 1$, where for $A,B \subset \Z^d$, $d(A,B)$ is the minimal $L^1$-distance between any two points of $A$ and $B$. One can find $i_0 \in \ov{\partial}^* \mclC_{i(x)}$ and $i_1 \in \ov{\partial}^* \mclC_{i(y)}$ such that $i_0 \sim i_1$. Both $i_0$ and $i_1$ are good, and by part (2) of Proposition~\ref{observ}, the intersection $\CC_{i_0} \cap \CC_{i_1}$ is not empty, and thus we have again $\hata(x,y) = 0$.

\noindent \emph{Case 3.} There remains to consider the case when $d(\ovmclC_{i(x)}, \ovmclC_{i(y)}) \ge 2$. Let 
$$
\tau_\infty = \inf \{ n \in \N : V(S_n) = +\infty \} .
$$
From part (4) of Proposition~\ref{observ}, we know that conditionally on the event
\begin{equation}
\label{defmclE}
\mcl{E} = \{ H_y < +\infty\} \cap \{ H_y \le \tau_\infty \},	
\end{equation}
we have $\tau_g(x) < H_y$ $\PP_x$-a.s. This means that, in order to go from $x$ to $y$ through livable sites only, one must visit a site of $\Delta^g(x)$. By symmetry, one must also visit a site of $\Delta^g(y)$ prior to reaching $y$, and this must happen after the first visit to $\Delta^g(x)$. Recalling that we write $\Theta_t$ for the time shift $(\Theta_t S)_n = S_{n+t}$, we have
\begin{eqnarray*}
e^{-a(x,y)} & = & \EE_x\left[ \exp\left( -\sum_{n = 0}^{H_y - 1} V(S_n) \right), H_y < + \infty \right] \\
& = & \EE_x\left[ \exp\left( -\sum_{n = 0}^{H_y - 1} V(S_n) \right), \mcl{E} \right] \\
& \le & \EE_x\left[ \exp\left( -\sum_{n = \tau_g(x)}^{H_y - 1} V(S_n) \right) \ \1_\mcl{E} \circ \Theta_{\tau_g(x)} \right] \\
& \stackrel{(\text{Markov})}{\le} & \EE_x \EE_{S_{\tau_g(x)}} \left[ \exp\left( -\sum_{n = 0}^{H_y - 1} V(S_n) \right), \mcl{E} \right] \\
& \le & \max_{x' \in \Delta^g(x)} \EE_{x'} \left[ \exp\left( -\sum_{n = 0}^{H_y - 1} V(S_n) \right), \mcl{E} \right] \\
& \le & \max_{x' \in \Delta^g(x)} \EE_{x'} \left[ \exp\left( -\sum_{n = 0}^{\tau_g(y) - 1} V(S_n) \right), \tau_g(y) < + \infty \right],
\end{eqnarray*}
the last term being precisely $e^{-\hata(x,y)}$.

We now consider the second inequality in (\ref{triangleineq}). For a path $\gamma = (\gamma_0,\ldots,\gamma_l)$, recall that we write ``$x \in \gamma$'' if there exists $k \le l$ such that $x = \gamma_k$. We write ``$S = \gamma$'' if for every $k \le l$, one has $S_k = \gamma_k$. 

If it is not possible to connect $x$ to $\Delta^g(x)$, or $y$ to $\Delta^g(y)$, through a sequence of livable sites, then $u(x)$ or $u(y)$ is infinite, and there is nothing to prove. So we now assume that such paths $\gamma_x$, $\gamma_y$ do exist.

\noindent \emph{Cases 1-2.} If $d(\mclC_{i(x)},\mclC_{i(y)}) \le 1$, then we have seen in the beginning of this proof that $\hata(x,y) = 0$ and that $\Delta^g(x) \cap \Delta^g(y)$ is not empty. Let $z$ be a point in this intersection. By part (5) of Proposition~\ref{observ}, we know that $\Delta^g(x)$ and $\Delta^g(y)$ are $\Z^d$-connected sets, so there exists a path $\gamma^1$ connecting the end-point of $\gamma_x$ to $z$ that stays inside $\Delta^g(x)$, and a path $\gamma^2$ connecting $z$ to the end-point of $\gamma_y$ that stays inside $\Delta^g(y)$. By concatenating $\gamma_x$, $\gamma^1$, $\gamma^2$ and (reversed $\gamma_y$), we get a path that connects $x$ to $y$ and visits only points of $\gamma_x \cup \Delta^g(x) \cup \Delta^g(y) \cup \gamma_y$ (where in the last expression, we identify $\gamma_x$ with the set of points visited by $\gamma_x$). We suppress possible loops to get a simple path $\gamma$ that conserves this property. Note that
\begin{eqnarray*}
e^{-a(x,y)} & = & \EE_x\left[ \exp\left( -\sum_{n = 0}^{H_y - 1} V(S_n) \right), H_y < + \infty \right] \\
& \ge & \EE_x\left[ \exp\left( -\sum_{n = 0}^{H_y - 1} V(S_n) \right), S = \gamma \right] \\
& \ge & e^{-u(x) + u(y)},
\end{eqnarray*}
and we thus obtain that $a(x,y) \le u(x) + u(y) = \hata(x,y) + u(x) + u(y)$.

\noindent \emph{Case 3.} If $d(\mclC_{i(x)},\mclC_{i(y)}) \ge 2$, let $x' \in \Delta^g(x)$. We can construct a path that is the concatenation of $\gamma_x$ and a path that connects the endpoint of $\gamma_x$ to $x'$ while staying inside $\Delta^g(x)$, and make a simple path $\gamma_{x,x'}$ out of it. Similarly, for any $y' \in \Delta^g(y)$, one can construct a simple path $\gamma_{y',y}$ that connects $y'$ to $y$ and visits only points in $\gamma_y \cup \Delta^g(y)$.

Observe that
\begin{eqnarray*}
e^{-a(x,y)} & = & \EE_x\left[ \exp\left( -\sum_{n = 0}^{H_y - 1} V(S_n) \right), H_y < + \infty \right] \\
& \ge & \EE_x\left[ \exp\left( -\sum_{n = 0}^{H_y - 1} V(S_n) \right), S = \gamma_{x,x'}, H_y < + \infty \right].
\end{eqnarray*}
Using the Markov property at the entrance time of the walk at $x'$, we obtain that the latter is equal to 
$$
\exp\left( -\sum_{z \in \gamma_{x,x'} \setminus \{x'\}} (V(z) + \log(2d)) \right) \EE_{x'}\left[ \exp\left( -\sum_{n = 0}^{H_y - 1} V(S_n) \right), H_y < + \infty \right].
$$
The first exponential term is larger than $\exp(-u(x))$. Concerning the second term, we have
$$
\EE_{x'}\left[ \exp\left( -\sum_{n = 0}^{H_y - 1} V(S_n) \right), H_y < + \infty \right] 
 = \EE_{x'}\left[ \exp\left( -\sum_{n = 0}^{H_y - 1} V(S_n) \right), \mcl{E} \right],
$$
where $\mcl{E}$ is the event defined in (\ref{defmclE}). Similarly, the Markov property enables us to rewrite the latter as
\begin{equation}
\label{compmarkov}
\EE_{x'}\left[ \exp\left( -\sum_{n = 0}^{\tau_g(y) - 1} V(S_n) \right)
\underbrace{\EE_{S_{\tau_g(y)}}\left[ \exp\left( -\sum_{n = 0}^{H_y - 1} V(S_n) \right), \mcl{E} \right]}
, \tau_g(y) < \infty \right].
\end{equation}
The expectation that is underbraced above is larger than $\exp(-u(y))$. Indeed, for any $y' \in \Delta^g(y)$, one has
\begin{eqnarray*}
\EE_{y'}\left[ \exp\left( -\sum_{n = 0}^{H_y - 1} V(S_n) \right), \mcl{E} \right] & = & \EE_{y'}\left[ \exp\left( -\sum_{n = 0}^{H_y - 1} V(S_n) \right), H_y < + \infty \right] \\
& \ge & \EE_{y'}\left[ \exp\left( -\sum_{n = 0}^{H_y - 1} V(S_n) \right), S = \gamma_{y',y} \right] \\
& \ge & \exp(-u(y)).
\end{eqnarray*}
The quantity appearing in (\ref{compmarkov}) is thus larger than $\exp(-\td{a}(x',y) - u(y))$. We have shown that, for any $x' \in \Delta_g(x)$, one has
$$
a(x,y) \le \td{a}(x',y) + u(x) + u(y).
$$
Taking the minimum over $x'$ finishes the proof.
\end{proof}
The aim of the next two propositions is to verify that the family $(\hata(mx,nx))_{m<n}$ satisfies the conditions required to apply the subadditive ergodic theorem. 
\begin{prop}
\label{p:triangle}
Let 
\begin{equation}
\label{defv}
v(x) = \sum_{z \in \Delta^g(x)} (V(x) + \log(2d)).
\end{equation}
For any $x,y,z \in \Z^d$, one has the approximate subadditivity property
$$
\hata(x,z) \le \hata(x,y) + v(y) + \hata(y,z) .
$$
\end{prop}
\begin{proof}
To lighten the notation, we write $\1_g(x)$ for $\1_{\{\tau_g(x) < +\infty  \}}$. Let $x'$ be any element of $\Delta^g(x)$, and observe that
\begin{equation*}
\begin{array}{l}
e^{-\td{a}(x',z)} \\
\displaystyle{ \quad =  \EE_{x'} \left[ \exp\left( -\sum_{n = 0}^{\tau_g(z) - 1} V(S_n) \right)\1_g(z) \right]} \\
\displaystyle{ \quad  \ge \EE_{x'} \left[ \exp\left( -\sum_{n = 0}^{\tau_g(z) - 1} V(S_n) \right) \1_g(y) \1_g(z) \circ \Theta_{\tau_g(y)} \right] }\\
\displaystyle{ \quad  \ge  \EE_{x'} \left[ \exp\left( -\sum_{n = 0}^{\tau_g(y) - 1} V(S_n) \right) \underbrace{\EE_{S_{\tau_g(y)}}\left[ \exp\left( -\sum_{n = 0}^{\tau_g(z) - 1} V(S_n) \right) \1_g(z)   \right]} \1_g(y) \right]},
\end{array}
\end{equation*}
where we used the Markov property on the last step. Let us see that the quantity underbraced above is almost surely larger than $e^{ - v(y) - \hat{a}(y,z)}$. Let $y'$ be any element of $\Delta^g(y)$ (recall that this set contains $S_{\tau_g(y)}$ a.s. if $\tau_g(y) < \infty$), and let $\ov{y} \in \Delta^g(y)$ be such that $\td{a}(\ov{y},z)$ is minimal, that is, $\ov{y}$ satisfying $\td{a}(\ov{y},z) = \hat{a}(y,z)$. Part (5) of Proposition~\ref{observ} ensures that $\Delta^g(y)$ is a $\Z^d$-connected set, so let $\gamma$ be a simple nearest-neighbour path contained in $\Delta^g(y)$ and linking $y'$ to $\ov{y}$. We have
\begin{multline*}
\EE_{y'}\left[ \exp\left( -\sum_{n = 0}^{\tau_g(z) - 1} V(S_n) \right) \1_g(z)   \right] \\
\ge \EE_{y'}\left[ \exp\left( -\sum_{n = 0}^{\tau_g(z) - 1} V(S_n) \right) \1_g(z), S = \gamma   \right].
\end{multline*}
Applying the Markov property at the time when $S$ hits the point $\ov{y}$, we obtain that the latter is equal to
$$
\exp\left( -\sum_{y'' \in \gamma \setminus \{\ov{y}\}} (V(y'') + \log(2d)) \right) e^{-\td{a}(\ov{y},z)} \ge e^{ - v(y) - \hat{a}(y,z)},
$$
so we have shown that, for any $x' \in \Delta^g(x)$,
$$
e^{-\td{a}(x',z)} \ge e^{-\td{a}(x',y) - v(y) - \hat{a}(y,z)},
$$
or equivalently, that
\begin{equation}
\label{firststepsubad}
\td{a}(x',z) \le \td{a}(x',y) + v(y) + \hat{a}(y,z).	
\end{equation}
We arrive at the conclusion of the proposition after a minimization over $x'$.
\end{proof}
\begin{prop}
\label{integrability}	
There exists $c > 0$ such that for any $x \in \Z^d$ and any $t \ge 0$, one has
$$
\P[v(x) \ge t] \le e^{-ct^{1-1/d}}.
$$
Moreover, for any $x \in \Z^d$ and any nearest-neighbour path $\gamma$ connecting $0$ to $x$, one has
$$
\hata(0,x) \le \sum_{z \in \gamma} v(z).
$$
\end{prop}
\begin{proof}
Recalling that any site in $\Delta^g(x)$ is healthy, we have
\begin{equation}
\label{upperboundv}
v(x) \le (M+\log(2d))|\Delta^g(x)| \le (M+\log(2d)) (2N+1)^d | \ovmclC_{i(x)} |,	
\end{equation}
so the first claim follows from Proposition~\ref{tailradius}. 

For the second claim, observe first that if $x$ and $y$ are neighbours, then clearly $d(\ovmclC_{i(x)},\ovmclC_{i(y)}) \le 1$, and we have seen in the proof of Proposition~\ref{compahata} that in this case, $\hata(x,y) = 0$. The second claim thus follows by iterating the inequality obtained in Proposition~\ref{p:triangle} along the path $\gamma$.
\end{proof}
\begin{proof}[Proof of part (3) of Theorem~\ref{t:existpoint}]
Let $x \in \Z^d \setminus \{ 0 \}$. Proposition~\ref{p:triangle} ensures that, for any $m,n,p \in \N$, 
$$
\hata(mx,nx) + v(nx) \le \hata(mx,px) + v(px) + \hata(px,nx) + v(nx),
$$
and $\hata(mx,nx) + v(nx)$ is integrable by Proposition~\ref{integrability}. The doubly indexed sequence $(\hata(mx,nx) + v(nx))_{m < n}$ is thus a stationary and integrable subadditive process. As a consequence, under assumption (H) only, 
$$
\frac{1}{n} \ll( \hata(0,nx) + v(nx) \rr)
$$
converges almost surely. We let $\hat{\alpha}(x)$ be the limit. By the same reasoning as in the proof of part (2) of Theorem~\ref{t:existpoint} (in section~\ref{s:ptpmoment}), we get that $\hat{\alpha}$ can be extended into a norm on $\R^d$ which satisfies
\begin{equation}
\label{lowerboundonalpha}
\hat{\alpha}(x) \ge -\|x\|_1 \log \E\ll[e^{-V(0)}\rr].	
\end{equation}

The random variables $(v(nx))_{n \in \N}$ being identically distributed and integrable, a Borel-Cantelli argument gives us that
$$
\frac{1}{n} \hata(0,nx) \xrightarrow[n \to + \infty]{\text{a.s.}} \hat{\alpha}(x).
$$
From Proposition~\ref{compahata}, we know that
$$
\ll| a(0,nx) - \hata(0,nx) \rr| \le u(0) + u(nx).
$$
The random variables $(u(nx))_{n \in \N}$ are identically distributed, and are almost surely finite when $V(0) < + \infty$ a.s., and thus
\begin{equation}
\label{aprob}
\frac{1}{n} a(0,nx) \xrightarrow[n \to + \infty]{\text{prob.}} \hat{\alpha}(x)	
\end{equation}
under this assumption. 
\end{proof}

\begin{rem}
Note that whenever the norm $\alpha$ introduced in section~\ref{s:ptpmoment} is well defined (that is, when $\E[Z(0)] < + \infty$), it follows from (\ref{aprob}) and part (2) of Theorem~\ref{t:existpoint} that $\hat{\alpha} = \alpha$. In other words, $\hat{\alpha}$ extends $\alpha$ to more general distributions of the potential, and we may as well put $\alpha = \hat{\alpha}$ whenever $\hat{\alpha}$ is well defined. With this definition, we have thus proved that under assumption (H) only, 
\begin{equation}
\label{hataasconv}
\frac{1}{n} \hata(0,nx) \xrightarrow[n \to + \infty]{\text{a.s.}} \alpha(x).
\end{equation}
\end{rem}

%
%
%
%
%
%
%
%
\section{Shape theorem without moment condition}
\label{s:shapenomoment}
\setcounter{equation}{0}
We now want to justify parts (2) and (3) of Theorem~\ref{t:shape}. Our first step is the derivation of a shape theorem for the set
$$
\hat{A}_t = \{x \in \Z^d : \hata(0,x) \le t \}.
$$
\begin{prop}
\label{p:shapehata}
Let $\hatA_t^\circ = \hatA_t + [0,1)^d$. Under assumption (H) only, for any $\eps > 0$, one has
$$
\forall \eps > 0, \quad \P[(1-\eps)K \subset t^{-1} \hatA_t^\circ \subset (1+\eps)K \text{ for all } t \text{ large enough} ]=  1	
$$
\end{prop}
The structure of the proof is very similar to what was done in section~\ref{s:shapemoment} for $A_{m,t}$. In fact, the only thing that we need is some equivalent of Proposition~\ref{controltail}.
\begin{prop}
\label{controltailH}
Under assumption (H) only, there exists $\mcl{K}, C > 0$ such that
$$
\P[\hata(0,x) \ge \mcl{K} \|x\|_1] \le e^{-C \|x\|_1}.
$$	
\end{prop}
\begin{proof}
Let us write $\mathbf{C}_\infty'$ for the (unique) infinite healthy cluster. It is proved in \cite[Theorem~1.1]{anpi} that there exists $C,\varrho > 0$ such that for any $z \in \Z^d$,
\begin{multline}
\label{e:anpi}
\P[0,z \in \mathbf{C}'_\infty \text{but no healthy path of length less than } \varrho \|z\|_1 \text{ connects } 0 \text{ to } z] \\
\le e^{-C\|z\|_1}.
\end{multline}

In particular, outside of an event of exponentially small probability (that is, of probability less than $e^{-C' \|x\|_1}$ for some $C' > 0$), any two points in $\mathbf{C}_\infty' \cap B(0,\|x\|_1)$ can be connected by a healthy path of length smaller than $\varrho' \|x\|_1$ for some fixed $\varrho'$.

If $\ov{\mcl{C}}_{i(0)} \cap \ov{\mcl{C}}_{i(x)} \neq \emptyset$, then we have seen in the beginning of the proof of Proposition~\ref{compahata} that $\hata(0,x) = 0$. Otherwise, one can find $y \in \Delta^g(0) \cap B(0,\|x\|_1)$ and $x' \in \Delta^g(x) \cap B(0,\|x\|_1)$. To see this, one can consider a nearest-neighbour path staying inside $B(0,\|x\|_1)$ that links $0$ to $x$, and observe that this path must pass through a point of $\Delta^g(0)$ and a point of $\Delta^g(x)$.

Note also that 
\begin{equation}
\label{subsetcinfty}
\bigcup_{i \in \mcl{C}_\infty} \CC_i \subset \mathbf{C}_\infty'.
\end{equation}
Indeed, it is clear that the union on the l.h.s.\ is an infinite connected component of healthy sites, by virtue of part (2) of Proposition~\ref{observ}, and uniqueness of the infinite healthy cluster thus ensures that \eqref{subsetcinfty} holds. Moreover, considering \eqref{defDelta}, any point in $\Delta^g(x)$ is in the set on the l.h.s.\ of \eqref{subsetcinfty}. In particular, $x'$ is in $\mathbf{C}_\infty'$, and by the same argument, we also have $y \in \mathbf{C}_\infty'$.

Outside of an event of exponentially small probability, $y$ and $x'$ can thus be connected by a healthy path $\gamma$ of length smaller than $\varrho' \|x\|_1$. By forcing the walk to follow this path, we obtain the inequality
$$
\hata(0,x) \le \td{a}(y,x) \le \sum_{z \in \gamma} (V(z) + \log(2d)) \le |\gamma|(M+\log(2d)),
$$
where we used the fact that $\gamma$ is a healthy path in the last step. This proves the proposition, since $|\gamma| \le \varrho' \|x\|_1$.
\end{proof}
\begin{proof}[Proof of Proposition~\ref{p:shapehata}]
It is more convenient to work with 
$$
\ov{a}(x,y) = \hata(x,y) + v(y),
$$
since we learn from Proposition~\ref{p:triangle} that it satisfies, for any $x,y,z \in \Z^d$,
$$
\ov{a}(x,z) \le \ov{a}(x,y) + \ov{a}(y,z).
$$
We let 
$$
\ov{A}_t = \{x \in \Z^d : \ov{a}(0,x) \le t\},
$$
and $\ov{A}_t^\circ = \ov{A}_t + [0,1)^d$. In order to reproduce the arguments of Propositions~\ref{coverball} and~\ref{shape2borne} for $\ov{a}$ instead of $a_m$, the only necessary information we need is that there exists $\mcl{K}$ such that
$$
\sum_{x \in \Z^d} \P[\ov{a}(0,x) \ge \mcl{K} \|x\|_1] < + \infty.
$$
This is ensured by Propositions~\ref{integrability} and \ref{controltailH}.
We thus obtain that, for any $\eps > 0$, one has
$$
\P[(1-\eps)K \subset t^{-1} \ov{A}_t^\circ \subset (1+\eps)K \text{ for all } t \text{ large enough} ]=  1.
$$
One can then proceed as in parts (3) of Propositions~\ref{coverball} and~\ref{shape2borne} to finish the proof of Proposition~\ref{p:shapehata}.
\end{proof}
\begin{proof}[Proof of part (2) of Theorem~\ref{t:shape}]
First, as $\hata(0,x) \le a(0,x)$ (see Proposition~\ref{compahata}), one has $A_t \subset \hatA_t$, so it follows from Proposition~\ref{p:shapehata} that for any $\eps > 0$, with probability one,
\begin{equation}
\label{infas}
t^{-1} A_t^\circ \subset (1+\eps) K \quad \text{ for all } t \text{ large enough}.	
\end{equation}
For any $\eps > 0$, we decompose 
\begin{multline}
\label{decomplebesgue}
| (t^{-1} A_t^\circ) \ \triangle \ K | \le \\
 | (t^{-1} A_t^\circ) \setminus (1+\eps)K| +  |(1+\eps)K \setminus (1-\eps)K| + |(1 - \eps) K \setminus (t^{-1} A_t^\circ)|.
\end{multline}
The first term in the r.h.s.\ of (\ref{decomplebesgue}) tends to $0$ a.s.\ as $t$ tends to infinity due to (\ref{infas}), while the second term can be made as small as desired by taking $\eps$ small enough. It thus suffices to show that, for any $\eps > 0$, one has
\begin{equation}
\label{sufficelebesgue}
|(1 - \eps) K \setminus (t^{-1} A_t^\circ)| \xrightarrow[t \to +\infty]{\text{a.s.}} 0.	
\end{equation}
We have shown in Proposition~\ref{compahata} that $a(0,y) \le \hata(0,y) + u(0)+u(y)$, so
$$
(1 - \eps) K \setminus (t^{-1} A_t^\circ) \subset \{x \in \R^d : \alpha(x) \le 1-\eps \text{ and } \hata(0,tx) + u(0) + u(tx) > t \},
$$
where we understand $\hata(0,tx)$ and $u(tx)$ to be respectively $\hata(0,\lfloor tx \rfloor)$ and $u(\lfloor tx \rfloor)$ if $tx \notin \Z^d$.
Since $u(0) < +\infty$ a.s.,\ it follows from Proposition~\ref{p:shapehata}  that almost surely, for all $t$ large enough, one has
$$
\ll\{x \in \R^d : \alpha(x) \le 1-\eps, u(0) + \hata(0,tx) > \ll(1-\frac{\eps}{2}\rr) t \rr\} = \emptyset.
$$
In order to prove (\ref{sufficelebesgue}), it is thus sufficient to show that
\begin{equation}
\label{presquebon}
|\{x \in \R^d : \alpha(x) \le 1-\eps \text{ and } u(tx) \ge \eps t/2 \}|\xrightarrow[t \to +\infty]{\text{a.s.}} 0. 
\end{equation}
If $C$ is large enough so that $K \subset [-C,C]^d$, then we can bound the l.h.s.\ of (\ref{presquebon}) by
$$
t^{-d} \sum_{z \in [-Ct,Ct]^d \cap \Z^d} \1_{\{u(z) \ge \eps t/2\}}.
$$
The multi-dimensional ergodic theorem (see \cite[Chapter~6]{krengel}) ensures that, for any constant $\mfk{m} > 0$, one has
$$
t^{-d} \sum_{z \in [-Ct,Ct]^d \cap \Z^d} \1_{\{u(z) \ge \mfk{m}\}} \xrightarrow[t \to +\infty]{\text{a.s.}} (2C)^d \ \P[u(0) \ge \mfk{m}].
$$
As a consequence,
$$
\limsup_{t \to +\infty} \ t^{-d} \sum_{z \in [-Ct,Ct]^d \cap \Z^d} \1_{\{u(z) \ge \eps t/2\}}
$$
should be smaller than $(2C)^d \ \P[u(0) \ge \mfk{m}]$ for any $\mfk{m}$, so it is in fact $0$ a.s.,\ and this finishes the proof.
\end{proof}
\begin{proof}[Proof of part (3) of Theorem~\ref{t:shape}]
It is clear that the observation (\ref{infas}) obtained above still holds under assumption (H). Let $(\e_t)_{t \ge 0}$ be such that $\lim_{t \to \infty} \e_t = + \infty$ and $\lim_{t \to \infty} \e_t/t = 0$. Reasoning as in the proof of part (2), one can see that it suffices to show that on the event $0 \in \mathbf{C}_\infty$, one has
\begin{equation}
\label{sufficelebesgue2}
|(1 - \eps) K \setminus (t^{-1} A_t^{\e_t})| \xrightarrow[t \to +\infty]{\text{a.s.}} 0.	
\end{equation}
On the event $0 \in \mathbf{C}_\infty$, it is still true that $u(0) < +\infty$ almost surely, so it is in fact sufficient to prove that
\begin{equation}
\label{presquebon2}
|\{x \in \R^d : \alpha(x) \le 1-\eps \text{ and } \forall y \text{ s.t. } \|y-tx\|_2 \le \e_t, \ u(y) \ge \eps t/2 \}|\xrightarrow[t \to +\infty]{\text{a.s.}} 0. 
\end{equation}
Taking $C$ large enough so that $K \subset [-C,C]^d$, we can bound the l.h.s.\ of (\ref{presquebon2}) by
\begin{equation}
\label{presquepresque}
t^{-d} \ \sum_{z \in [-Ct,Ct]^d \cap \Z^d} \1_{\{ \forall y \text{ s.t. } \|y-z\|_2 \le \e_t, \ u(y) \ge \eps t/2\}} \xrightarrow[t \to +\infty]{\text{a.s.}} 0. 
\end{equation}
For any two constants $\mfk{m}_1, \mfk{m}_2 < +\infty$, the multi-dimensional ergodic theorem ensures that
\begin{multline*}
t^{-d} \ \sum_{z \in [-Ct,Ct]^d \cap \Z^d} \1_{\{ \forall y \text{ s.t. } \|y-z\|_2 \le \mfk{m}_1, \ u(y) \ge \mfk{m}_2\}} \\
\xrightarrow[t \to +\infty]{\text{a.s.}} (2C)^d \ \P[\forall y \text{ s.t. } \|y\|_2 \le \mfk{m}_1, \ u(y) \ge \mfk{m}_2].
\end{multline*}
The limsup as $t \to +\infty$ of the r.h.s.\ of (\ref{presquepresque}) must thus be smaller than 
\begin{equation}
\label{tsol}
(2C)^d \ \P[\forall y \text{ s.t. } \|y\|_2 \le \mfk{m}_1, \ u(y) \ge \mfk{m}_2]	
\end{equation}
for any $\mfk{m}_1, \mfk{m}_2$. The proposition will thus be proved if we can show that the probability in (\ref{tsol}) can be made arbitrarily small by a suitable choice of $\mfk{m}_1$ and $\mfk{m}_2$. Noting that $u(y)$ is not a.s.\ equal to $+\infty$, we can let $\mfk{m}_2$ be such that $\P[u(y) < \mfk{m}_2] > 0$. By the ergodic theorem, one has
\begin{equation}
\label{cvas}
(2n+1)^{-d} \ \sum_{y \in B(0,n)} \1_{\{u(y) < \mfk{m}_2\}} \xrightarrow[n \to + \infty]{\text{a.s.}} \P[u(y) < \mfk{m}_2] > 0.
\end{equation}
The conclusion now follows observing that
\begin{equation*}
\P[\exists y : \|y\|_\infty \le n \text{ and } u(y) < \mfk{m}_2] = \P\ll[ \sum_{y \in B(0,n)} \1_{\{u(y)  < \mfk{m}_2\}} \ge 1 \rr]  \xrightarrow[n \to + \infty]{}  1.
\end{equation*}
\end{proof}
%
%
%
%
%
%
%
%
\section{Point-to-hyperplane exponent}
\label{s:pth}
\setcounter{equation}{0}
\begin{proof}[Proof of Theorem~\ref{t:existplane}]
Let $x \in \R^d \setminus \{0\}$. We start by showing that, on the event  $0 \in \mathbf{C}_\infty$,
\begin{equation}
\label{limsupa*}
\limsup_{t \to + \infty} t^{-1} \ a^*(x,t) \le \frac{1}{\alpha^*(x)}.	
\end{equation}
Note first that
\begin{equation}
\label{dualdual}
\frac{1}{\alpha^*(x)} = \inf \{\alpha(y) \ | \ y \in \R^d, x \cdot y = 1 \}.	
\end{equation}
Let $y_0 \in \R^d$ be such that $x \cdot y_0 > 1$. We know from part (3) of Theorem~\ref{t:shape} that, on the event $0 \in \mathbf{C}_\infty$,
\begin{equation}
\label{conerge}
| \big(t^{-1} A_{\alpha(y_0)t}^{\sqrt{t}}\big) \ \triangle \ \big(\alpha(y_0) K\big) | \xrightarrow[t \to +\infty]{\text{a.s.}} 0.
\end{equation}
Since $K$ is the unit ball of some norm, $\alpha(y_0) \neq 0$, and $y_0 \in \alpha(y_0) K$, one can see that the set
$$
\{y \in \R^d : \|y-y_0\|_2 \le \eps \} \cap \big(\alpha(y_0) K\big)
$$
has non-zero Lebesgue measure. In particular, it follows from \eqref{conerge} that on the event $0 \in \mathbf{C}_\infty$, for all $t$ large enough and any $\eps > 0$, one has
$$
\big(t^{-1} A_{\alpha(y_0)t}^{\sqrt{t}}\big) \cap \{y : \|y-y_0\|_2 \le \eps \} \neq \emptyset.
$$
As a consequence, for all $t$ large enough, there exists $y_t, y_t'$ such that
$$
y_t \in A_{\alpha(y_0)t}, \quad \|y_t - y_t'\|_2 \le \sqrt{t}, \quad \|y_t' - t y_0\|_2 \le \eps t.
$$
For $t$ large enough, one must have $\|y_t - t y_0\|_2 \le 2 \eps t$. For such $t$, one has
$$
x \cdot y_t  = t  x \cdot y_0 + x \cdot (y_t - t y_0) \ge t (x \cdot y_0 - 2 \eps \|x\|_2).
$$
Choosing $\eps$ such that $x \cdot y_0 - 2 \eps \|x\|_2 \ge 1$ ensures that for $t$ large enough, $y_t$ lies in the half-space
$$
\{z \in \R^d :  x \cdot z \ge t\},
$$
and it follows that $a^*(x,t) \le a(0,y_t)$. Besides, $y_t$ belongs to $A_{\alpha(y_0) t}$, so in fact $a^*(x,t) \le \alpha(y_0) t$. We have proved that, on the event $0 \in \mathbf{C}_\infty$, for any $y_0$ such that $x \cdot y_0 > 1$, one has
$$
\limsup_{t \to +\infty} t^{-1} \ a^*(x,t) \le \alpha(y_0).
$$
In view of (\ref{dualdual}), this implies (\ref{limsupa*}).

Let us now prove that, with probability one,
\begin{equation}
\label{liminfa*}
\liminf_{t \to + \infty} t^{-1} \ a^*(x,t) \ge \frac{1}{\alpha^*(x)}.	
\end{equation}
In order to do so, we adapt the argument of \cite[Corollary~5.2.11]{sznitman} to our discrete setting. Let $x_2,\ldots, x_d$ be such that $x,x_2,\ldots, x_d$ is an orthogonal basis of $\R^d$. For $L > 0$, let $P_L$ be the parallelepiped defined by
$$
P_L = \big\{y \in \R^d : -L \le x \cdot y \le 1 \text{ and } \forall i \in \{2 , \ldots, d\}, |x_i \cdot y| \le L \big\}.
$$
Since $\alpha$ is a norm and using (\ref{dualdual}), one can check that for $L$ large enough, one has
\begin{equation}
\label{infPL}
\frac{1}{\alpha^*(x)} = \inf_{\partial P_L} \alpha,
\end{equation}
where $\partial P_L$ denotes the boundary of $P_L$, in the continuous sense. We fix $L$ satisfying this property. Clearly, the random walk starting from $0$ must reach a point of $\underline{\partial} (t P_L)$ (the inner boundary of $(t P_L) \cap \Z^d$) before reaching the half space
$$
\{z \in \R^d :  x \cdot z \ge t\},
$$
so in particular, if we let $T_t$ be the entrance time in $\underline{\partial} (t P_L)$, we have
\begin{eqnarray*}
e^{-a^*(x,t)} & \le & \EE_0\left[ \exp\left( -\sum_{n = 0}^{T_t - 1}  V(S_n) \rr) \ \1_{\{T_t < +\infty\}} \right] \\
& \le & \ll| \underline{\partial} (t P_L) \rr| \exp\ll(-\min_{z \in \underline{\partial} (t P_L)} a(0,z)\rr).
\end{eqnarray*}
As $\ll| \underline{\partial} (t P_L) \rr|$ grows polynomially with $t$, inequality (\ref{liminfa*}) will be proved if we can show that
\begin{equation}
\label{liminfsurlebord}
\liminf_{t \to + \infty} \ t^{-1}  \min_{z \in \underline{\partial} (t P_L)} a(0,z) \ge \inf_{\partial P_L} \alpha. 	
\end{equation}
To see that this is true, consider sequences $t_n \to + \infty$ and $z_n \in \underline{\partial} (t_n P_L)$. It suffices to show that for any such sequence,
$$
\liminf_{n \to + \infty} \ t_n^{-1}  a(0,z_n) \ge \inf_{\partial P_L} \alpha.
$$
Note that $z_n/t_n$ is bounded, so possibly extracting a subsequence, we may as well assume that $z_n/t_n$ converges to some $z^* \in \partial P_L$. From part (3) of Theorem~\ref{t:shape} and Remark~\ref{rem1}, we know that
$$
\liminf_{n \to + \infty} \frac{a(0,z_n)}{\alpha(z_n)} \ge 1.
$$
But
$$
\frac{a(0,z_n)}{\alpha(z_n)} = \frac{a(0,z_n)}{t_n} \frac{1}{\alpha(z_n/t_n)},
$$
so we get that
$$
\liminf_{n \to + \infty} t_n^{-1} {a(0,z_n)} \ge \alpha(z^*),
$$
and, in view of (\ref{infPL}), this finishes the proof.
\end{proof}

%
%
%
%
%
%
%
%
\section{Large deviations: the lower bound}
\label{s:ldplow}
\setcounter{equation}{0}
Without loss of generality, we assume from now on that $\mathbf{0}$ \textbf{is in the support of the distribution of} $\mathbf{V(0)}$.

In order to derive a large deviation principle for the random walk under the weighted measure $\PP_{n,V}$, we need to consider properties of the motion in the potential $V_\lambda = \lambda + V$. Under this potential, a healthy site $x$ defined with reference to the potential $V$ is such that $V_\lambda(x) \le \lambda + M$. Since the specific value of $M$ plays no significant role, we keep the notion of healthy site attached to the potential $V$, and similarly, we keep the notion of a good macroscopic site, the constructions of $\mclC_\infty$, $\mclC_i$, $\Delta^g(x)$, $\tau_g(x)$, independent of~$\lambda$. On the other hand, extending the definitions in (\ref{defa}), (\ref{defalpha}), (\ref{defAt}), (\ref{defK}), (\ref{deftda}), (\ref{defhata}), (\ref{defv}), we let $a_\lambda$, $\alpha_\lambda$, $A_{\lambda,t}$, $K_\lambda$, $\td{a}_\lambda$, $\hata_\lambda$, $v_\lambda$, be, respectively, the quantities $a$, $\alpha$, $A_{t}$, $K$, $\td{a}$, $\hata$, $v$, obtained when the potential $V$ is replaced by $V_\lambda$. For instance, one has
\begin{equation}
\label{deftdalambda}
\td{a}_\lambda(x,y) = - \log \EE_{x} \left[ \exp\left( -\sum_{n = 0}^{\tau_g(y) - 1} (\lambda + V)(S_n) \right), \tau_g(y) < + \infty \right],	
\end{equation}
where we stress that $\tau_g(y)$ does not depend on $\lambda$.

The following proposition and its proof are similar to \cite[Proposition~17]{zer}.
\begin{prop}
\label{partitionfunction}
On the event $0 \in \mathbf{C}_\infty$, the partition function defined in (\ref{defZVn}) satisfies
$$
\frac{1}{n} \log Z_{n,V} \xrightarrow[n \to + \infty]{\text{a.s.}} 0.
$$
\end{prop}
A notion that will be useful to prove the proposition, and also later on, is that of ``clearings''. For $\eps > 0$ and $R \in \N$, we say that a site $x \in \Z^d$ is an $(\eps,R)$-\emph{clearing} if 
\begin{equation}
\label{defclearing}
\forall y \in B(x,R),  \ V(y) \le \eps.	
\end{equation}
\begin{proof}[Proof of Proposition~\ref{partitionfunction}]
Let $\eps > 0$ and $R \in \N$. On the event $0 \in \mathbf{C}_\infty$, there exists a path of livable sites going from $0$ to an $(\eps,R)$-clearing (we will in fact give a control on the distance from the origin to a clearing in Proposition~\ref{clearings}). We write $x$ and $\gamma$ for, respectively, the clearing and the path connecting $0$ to $x$. We obtain a lower bound on $Z_{n,V}$ by forcing the walk to travel along the path $\gamma$ and then stay within distance $R$ from $x$:
$$
Z_{n,V} \ge \exp\ll( -\sum_{z \in \gamma \setminus \{x\}} (V(z) + \log(2d)) \rr) e^{-\eps n} \PP_x\ll[ \forall k \le n, \ S_k \in B(x,R) \rr].
$$
The first exponential term does not depend on $n$, so
\begin{equation}
\label{proofofpf}	
\limsup_{n \to + \infty} \ - \frac1n \log Z_{n,V} \le \eps + \limsup_{n \to +\infty} \ - \frac1n \log \PP_0\ll[ \forall k \le n, \ S_k \in B(0,R) \rr].
\end{equation}
The limsup in the r.h.s.\ of (\ref{proofofpf}) can be bounded by noting that
$$
\forall k \le n, \ S_k \in B(0,R)
$$
is implied by the event that $S_{kR} = 0$ for any $k < n/R$, and thus
\begin{equation}
\label{restela}
\limsup_{n \to +\infty} \ - \frac1n \log \PP_0\ll[ \forall k \le n, \ S_k \in B(0,R) \rr] \le -\frac{1}{R} \log \PP_0[S_R = 0].	
\end{equation}
The probability $\PP_0[S_R = 0]$ decays only polynomially with $R$ as $R$ tends to infinity, so the limsup in the l.h.s.\ of (\ref{proofofpf}) can be made arbitrarily small by taking $\eps$ small enough and $R$ large enough, thus completing the proof.
\end{proof}
\begin{prop}
\label{p:jointconti}
The function $(\lambda,x) \mapsto \alpha_\lambda(x)$ is concave increasing in $\lambda$ for fixed~$x$, and jointly continuous.
\end{prop}
\begin{proof}
The concavity of $\lambda \mapsto \alpha_\lambda(x)$ follows from Hölder's inequality applied to formula (\ref{deftdalambda}), while the fact that it is increasing (in the wide sense) is obvious. Let us see that it is also continuous in the $\lambda$ variable. As a concave function, it is lower semi-continuous. On the other hand, recall that we built $\alpha_\lambda(x)$ by applying the subadditive ergodic theorem to 
\begin{equation}
\label{defova}
\ov{a}_\lambda(x,y) = \hata_\lambda(x,y) + v_\lambda(y),	
\end{equation}
see the proof of part (3) of Theorem~\ref{t:existpoint} in the end of section~\ref{s:ptpnomoment}. It thus follows that
\begin{equation}
\label{alphalambdainf}
\alpha_\lambda(x) = \inf_{n \in \N} \frac{1}{n} \E\ll[ \ov{a}_\lambda(0,nx) \rr].	
\end{equation}
The function $\lambda \mapsto \E\ll[\ov{a}_\lambda(0,nx) \rr]$ is continuous, since $\lambda \mapsto \ov{a}_\lambda(0,nx)$ is monotone and $\ov{a}_\lambda(0,nx)$ is integrable for any $\lambda$. As an infimum of continuous functions, the function $\lambda \mapsto \alpha_\lambda(x)$ is upper semi-continuous, and thus in fact continuous.

Finally, observe that $x \mapsto \alpha_\lambda(x)$ are continuous functions, and that they form a monotone family of functions. So Dini's theorem can be used to obtain the joint continuity in the variables $(\lambda,x)$ over any compact set.
\end{proof}
\begin{prop}
\label{p:unifliminf}
With probability one, the following holds for any $\lambda \ge 0$:
$$
\liminf_{\|x\|_1 \to + \infty} \frac{a_\lambda(0,x)}{\alpha_\lambda(x)} \ge 1.
$$
\end{prop}
\begin{proof}
In view of Remark~\ref{rem1}, it suffices to check that with probability one, one has for any $\lambda \ge 0$ and any $\eps > 0$ that
\begin{equation}
\label{liminfalambda}
t^{-1} A_{\lambda,t} \subset (1+\eps)K_\lambda \text{ for all } t \text{ large enough}.	
\end{equation}
Due to part (3) of Theorem~\ref{t:shape}, we know that with probability one, (\ref{liminfalambda}) holds for any rational $\lambda$ and $\eps$. The monotonicity of property (\ref{liminfalambda}) with respect to $\eps$ enables to disregard the restriction over the values of $\eps$. Consider now an arbitrary $\lambda > 0$. For any $\delta > 0$, the set $A_{\lambda,t}$ is a subset of $A_{\lambda-\delta,t}$. The claim would thus be proved if we can show that, for any $\eps > 0$, any sufficiently small $\delta > 0$ is such that
$$
K_{\lambda-\delta} \subset (1+\eps) K_{\lambda}.
$$
But this follows from the joint continuity (turned into uniform continuity over a compact set) of $(\lambda,x) \mapsto \alpha_\lambda(x)$ proved in Proposition~\ref{p:jointconti}.
\end{proof}
\begin{rem}
There seems to be a minor problem with the analogous result obtained in \cite[Corollary~16]{zer}. A way to fix the claim is to impose $K_n$ to be a subset of $\Z^d$, and check that this weaker version is still sufficient for the proof of \cite[Theorem~19]{zer}.
\end{rem}
\begin{proof}[Proof of (\ref{ldclosed}) in Theorem~\ref{t:ldp}]
The argument is similar to the proof of \cite[Theorem~5.4.2]{sznitman}. First, let us note that the rate function $I$ is infinite outside of a compact set. Indeed, we learn from (\ref{lowerboundonalpha}) that
$$
\alpha_\lambda(x) \ge \ll(\lambda - \log \E\ll[ e^{-V(0)} \rr] \rr) \|x\|_1,
$$
so $I(x) = \sup_{\lambda \ge 0} (\alpha_\lambda(x) - \lambda)$ is infinite as soon as $\|x\|_1 > 1$. In order to prove (\ref{ldclosed}), it is thus sufficient to consider a compact $F \subset \R^d$. Let $T_n$ be the entrance time in $nF$. One has
\begin{eqnarray*}
Z_{n,V} \ \PP_{n,V}[S_n \in nF] & = & e^{\lambda n} \ \EE_0\ll[\exp\ll(-\sum_{k = 0}^{n-1} (\lambda + V)(S_k)  \rr),  S_n \in nF\rr] \\
& \le & e^{\lambda n} \ \EE_0\ll[\exp\ll(-\sum_{k = 0}^{T_n-1} (\lambda + V)(S_k)  \rr),  T_n < + \infty \rr] \\
& \le & e^{\lambda n} \ |nF\cap \Z^d| \exp\ll( -\min_{z \in nF\cap \Z^d} a_\lambda(0,z) \rr).
\end{eqnarray*}
Clearly, $|nF\cap \Z^d|$ grows polynomially with $n$. Moreover, we claim that on a set of full probability measure not depending on $F$ or $\lambda$, one has
$$
\liminf_{n \to +\infty} \ n^{-1} \min_{z \in nF\cap \Z^d} a_\lambda(0,z) \ge \inf_F \alpha_\lambda.	
$$
This can be shown to be true in the same way as was done for identity (\ref{liminfsurlebord}), with the help of Proposition~\ref{p:unifliminf}. Using also Proposition~\ref{partitionfunction}, we thus obtain that almost surely on the event $0 \in \mathbf{C}_\infty$, for any compact $F$ and any $\lambda\ge 0$, the following holds:
\begin{equation}
\label{halfwaylow}
\liminf_{n \to +\infty} \ -\frac1n \log \PP_{n,V}[S_n \in nF] \ge \inf_F \alpha_\lambda - \lambda .
\end{equation}
We could take the supremum over $\lambda$ in the r.h.s.\ of the above inequality, but supremum and infimum would not come ordered in the way we are hoping for. Let $\eps > 0$. The compact set $F$ is covered by the union over $\lambda$ of the open sets
$$
\{x \in \R^d : \alpha_\lambda(x) - \lambda > \inf_F I - \eps\},
$$
so we can find $\lambda_1,\ldots,\lambda_m$ such that the compact sets
$$
F_i = \{x \in F : \alpha_{\lambda_i}(x) - \lambda_i \ge \inf_F I - 2\eps \}
$$
cover $F$. Since $\PP_{n,V}[S_n \in nF] \le \sum_{i=1}^m \PP_{n,V}[S_n \in nF_i]$ and using (\ref{halfwaylow}), we thus obtain that
$$
\liminf_{n \to +\infty} \ -\frac1n \log \PP_{n,V}[S_n \in nF] \ge \inf_{1 \le i \le m} \ll( \inf_{F_i} \alpha_{\lambda_i} - \lambda_i \rr) \ge \inf_F I - 2\eps,
$$
which completes the proof.
\end{proof}

%
%
%
%
%
%
%
%
\section{Large deviations: the upper bound}
\label{s:ldpup}
\setcounter{equation}{0}
In contrast to point-to-point or point-to-hyperplane Lyapunov exponents, the large deviation principle involves events for which the random walk is asked to reach some part of the space \emph{at a specific time}. It is therefore necessary to understand the speed at which the random walk travels when forced to go in a particular direction. This motivates the introduction, for any $x, y \in \Z^d$ and $0 \le s_1 < s_2$, of the event $\mcl{V}_{y,s_1,s_2}$ defined by
$$
s_1 \le \tau_g(y) \le s_2,
$$
and of
$$
\td{a}_\lambda(x,y, s_1,s_2) = - \log \EE_{x} \left[ \exp\left( -\sum_{n = 0}^{\tau_g(y) - 1} (\lambda + V)(S_n) \right), \mcl{V}_{y,s_1,s_2} \right],
$$
$$
\hata_\lambda(x,y,s_1,s_2) = \min_{x' \in \Delta^g(x)} \td{a}_\lambda(x',y,s_1,s_2).
$$ 
The quantity 
$$
\lim_{n \to + \infty} \hata_\lambda(0,nx,ns_1,ns_2)
$$
should be understood as the cost of travelling in the direction of $x$ with an ``inverse speed'' contained in the interval $[s_1,s_2]$ (properly speaking, the inverse speed should be $s_1/\|x\|$ instead of $s_1$, but we forget about this).

We write $\alpha_{\lambda+}'(x)$ (resp.\ $\alpha_{\lambda-}'(x)$) for the right (resp.\ left) derivative of the concave function $\lambda \mapsto \alpha_\lambda(x)$. If these coincide, we write $\alpha'_\lambda(x)$ for their common value.

In a sense made specific in Proposition~\ref{velocity}, under the weighted measure, the random walk forced to travel in the direction of $x$ and in the potential $V_\lambda$ does so at inverse speed $\alpha'_\lambda(x)$. This will be a crucial information in order to devise ``strategies'' to make the random walk be at a specific place at time $n$. Another important ingredient is the possibility for the walk to wait in a favorable place (i.e.\ a clearing) if it arrives too early at a prescribed location. The next proposition ensures that there are always clearings nearby.

\begin{prop}
\label{clearings}
Recall the notion of $(\eps,R)$-clearing defined in (\ref{defclearing}), and let $\xi = 1-1/d$. For any $\eps > 0, R \in \N$, there exists $C, \varrho > 0$ such that	
$$
\P\ll[
\begin{array}{l}
\text{for any } z \in \Delta^g(0), \|z\|_1 \le n \text{ and there is an } (\eps,R)\text{-clearing}\\
\text{connected to } z \text{ through a healthy path of length at most } \varrho n
\end{array}
\rr]  \ge 1- e^{-Cn^{\xi}}.
$$
As a consequence, there exists $C' > 0$ such that for all but a finite number of $x$, every $z \in \Delta^g(x)$ is such that
$\|x-z\|_1 \le C' (\log \|x\|_1)^{1/\xi}$ and $z$ is connected to an $(\eps,R)$-clearing through a healthy path of length at most $C' (\log \|x\|_1)^{1/\xi}$.
\end{prop}
\begin{proof}
We recall that we write $\mathbf{C}'_\infty$ for the infinite healthy cluster. We start by proving that there exists $C > 0$ such that
\begin{equation}
\label{clearingisthere}
\P[\text{there is an } (\eps,R)\text{-clearing in } \mathbf{C}'_\infty \cap B(0,n)] \ge 1-e^{-Cn^{d/(d+1)}} \ge 1 - e^{-Cn^{\xi}}.
\end{equation}
We chop the box $B(0,n)$ into disjoint sub-boxes of size $n^\chi$, with $0 < \chi < 1$ to be determined. There may remain sites which do not fit into a sub-box contained in $B(0,n)$ due to boundary effects, we discard them. Let us write $\mcl{S}_n$ for the set of centres of the sub-boxes thus constructed. The cardinality of $\mcl{S}_n$ is at least $C n^{d(1-\chi)}$. For any $x \in \mcl{S}_n$, let $\mcl{A}_x$ be the event defined by:
$$
x \text { is an } (\eps,R)\text{-clearing and } \ov{\partial} B(x,R) \text{ is connected to } \underline{\partial} B(x,n^{\chi}) \text{ via a healthy path}.
$$
Since the probability for $x$ to belong to $\mathbf{C}_\infty'$ is non-zero, there exists $p_0 > 0$ such that $\P[\mcl{A}_x] \ge p_0$ uniformly over $n$. Hence, with high probability, there exists $x \in \mcl{S}_n$ such that $\mcl{A}_x$ happens, or more precisely,
\begin{equation}
\label{existsinSn}
1-\P\ll[\bigcup_{x \in S_n}\mcl{A}_x\rr] \le (1-p_0)^{|\mcl{S}_n|} \le \exp\ll(Cn^{d(1-\chi)} \log(1-p_0)\rr).
\end{equation}

On the other hand, it is very unlikely that in $B(0,n)$, a healthy cluster of size at least $n^\chi$ exists outside of $\mathbf{C}_\infty'$. Indeed, this probability is bounded by
$$
\sum_{x \in B(0,n)} \P[n^\chi \le \mathrm{diam}(\text{healthy cluster containing } x) < + \infty],
$$
which is smaller than $e^{-C n^\chi}/C$ for some $C > 0$ due to \cite[Theorem~8.21]{grim}. Taking $\chi = d/(d+1)$ yields claim (\ref{clearingisthere}).

Recall also that, due to Proposition~\ref{tailradius}, 
$$
\P[\Delta^g(0) \subset B(0,n)] \ge 1 - e^{-cn^{\xi}}.
$$

In order to conclude, it thus suffices to see that any two points of $\mathbf{C}'_\infty \cap B(0,n)$ can be connected via a short healthy path. Considering the bound \eqref{e:anpi}, it is clear that outside of an event of probability less than $e^{-Cn}$, any two points of $\mathbf{C}'_\infty \cap B(0,n)$ can be connected to one another through a path of length at most $\varrho n$, and this finishes the proof.
\end{proof}

Let us now give some useful properties of the function $I$.
\begin{prop}
\label{ratefunction}
\begin{enumerate}
\item The function $I$ is a convex lower semi-continuous good rate function.
\item The set 
$$
\mathrm{dom}(I) = \{ x \in \R^d : I(x) < + \infty \}
$$
is convex and contains a neighbourhood of the origin.
\item
For any open $O \subset \R^d$, one has 
$$
\inf_{O \cap \Q^d} I = \inf_O I.
$$
\end{enumerate}	
\end{prop}
\begin{proof}
Since $\alpha$ is a norm, it is clear that $x \to \alpha_\lambda(x)$ is convex, and thus so is $I$. As a supremum of continuous functions, $I$ is also lower semi-continuous. We have seen at the beginning of the proof of (\ref{ldclosed}) (in the previous section) that $I$ is infinite outside of a bounded set, so it is clearly a good rate function.

The convexity of the function $I$ implies the convexity of the set $\mathrm{dom}(I)$. Let $e_i$ be the $i$-th vector of the canonical basis of $\R^d$. Due to the concavity of $\lambda \mapsto \alpha_\lambda(e_i)$, there exists $C_i$ such that $\alpha_\lambda(e_i) \le C_i(\lambda +1)$, and from this one can check that $\pm e_i/(2C_i)$ must be in $\mathrm{dom}(I)$. Convexity then ensures that a full neighbourhood of the origin is in $\mathrm{dom}(I)$. In fact, we have proved that for $r > 0$ small enough, the function $I$ is uniformly bounded on $D(0,r)$.

For the last part, let $O$ be an open subset of $\R^d$, and $x \in O$. We will show that
\begin{equation}
\label{infO}
\inf_{O \cap \Q^d} I \le I(x).	
\end{equation}
If $x \notin \mathrm{dom}(I)$, then there is nothing to prove. Otherwise, let $\mcl{T}$ be the convex hull of $D(0,r) \cup \{x\}$. We claim that
\begin{equation}
\label{limhull}
\lim_{\substack{z \to x \\ z \in \mcl{T}}} I(z) = I(x).	
\end{equation}
Lower semi-continuity of $I$ guarantees that 
$$
\liminf_{{z \to x}} I(z) \ge I(x).	
$$
On the other hand, let $z_n \in \mcl{T}$ be a sequence tending to $x$ as $n$ tends to infinity. One can represent $z_n$ as $\mu_n x + \nu_n y_n$, where $\mu_n,\nu_n \ge 0$ satisfy $\mu_n + \nu_n \le 1$ and $\lim_{n \to \infty} \mu_n = 1$, and $y_n \in D(0,r)$. Convexity yields that
$$
I(z_n) \le \mu_n I(x) + \nu_n I(y_n).
$$
Since $I(y_n)$ is uniformly bounded, it follows that
$$
\limsup_{n \to +\infty} I(x_n) \le I(x),
$$
and thus (\ref{limhull}) is proved. To conclude, 
note that $x$ is in the closure of $\mcl{T} \cap \Q^d$, so one can find a sequence $x_n \in \mcl{T} \cap \Q^d$ such that 
$$
\lim_{n \to +\infty} x_n = x \text{  and } \lim_{n \to +\infty} I(x_n) = I(x).
$$
The set $O$ being open, the sequence $(x_n)$ is ultimately in $O$, thus justifying inequality (\ref{infO}).
\end{proof}

\begin{prop}
\label{alllambda}
With probability one, the following three properties hold for any $\lambda \ge 0$.
\begin{enumerate}
\item For any $x \in \Q^d$,
\begin{equation}
\label{alltriv}
\lim_{n \to +\infty} \frac{v_\lambda(nx)}{n} = 0.
\end{equation}
\item For any $x \in \Q^d$ and any $\rho_1,\rho_2 \in \Q$ satisfying $0\le \rho_1 < \rho_2$, 
\begin{equation}
\label{alllambda1}
\lim_{n \to + \infty} \frac{1}{n} \hata_\lambda(\rho_1nx,\rho_2nx) = (\rho_2 - \rho_1) \alpha_\lambda(x).
\end{equation}
\item For any $x \in \Q^d$, any $\rho_1,\rho_2 \in \Q$ satisfying $0\le \rho_1 < \rho_2$ and any sequence $x_n \in \Delta^g(\rho_1 nx)$, 
\begin{equation}
\label{alllambda2}
\lim_{n \to + \infty} \frac1n \td{a}_\lambda(x_n,\rho_2 nx) = (\rho_2 - \rho_1) \alpha_\lambda(x).
\end{equation}
\end{enumerate}
\end{prop}
\begin{rem}
This proposition is certainly not optimal. In fact, it may well be that a uniform shape theorem comparable to \cite[Theorem~13]{zer} holds. This would however require more effort, and as the present version is sufficient for our purpose, I did not try to pursue this question further.
\end{rem}
\begin{proof}
We first verify that the three claims hold true for any fixed $\lambda \ge 0$. 

The first claim is true due to the fact that $(v_\lambda(nx))_{n \in \N}$ are identically distributed and integrable random variables, see Proposition~\ref{integrability}.

Let $x \in \Q^d$. With the equivalent of Proposition~\ref{radiallimits} for $\hata_\lambda$, we already know that 
\begin{equation}
\label{radiallim2}
\frac1n \hata_\lambda(0,nx) \xrightarrow[n \to + \infty]{\text{a.s.}} \alpha(x).	
\end{equation}
(In fact, the shape theorem proved in Proposition~\ref{p:shapehata} gives us a stronger result, see Remark~\ref{rem1}.)
Claim (2) is thus true if $\rho_1 = 0$. Otherwise, it suffices to verify (\ref{alllambda1}) when $\rho_1 = 1$ and $\rho_2 > 1$ is an integer. Recall from Proposition~\ref{p:triangle} that
$$
\hata_\lambda(0,\rho_2 nx) \le \hata_\lambda(0,nx) + \hata_\lambda(nx,\rho_2 nx) + v_\lambda(nx).
$$
Using (\ref{alltriv}), we are led to
$$
\liminf_{n \to + \infty} \frac1n \hata_\lambda(nx, \rho_2 nx) \ge (\rho_2 - 1) \alpha(x).
$$
Let $k$ be an integer such that $kx \in \Z^d$. Using Proposition~\ref{p:triangle} iteratively, we obtain
\begin{multline}
\label{allez}
\hata_\lambda(nx,\rho_2nx) \le \sum_{j = \lceil n/k \rceil}^{\lfloor \rho_2 n/k \rfloor-1} \ov{a}_\lambda\big(jkx,(j+1)kx\big)  \\
+ \ov{a}_\lambda\big(nx, \lceil n/k \rceil kx) + \hata_\lambda\big(\lfloor \rho_2n/k \rfloor kx, \rho_2nx\big),
\end{multline}
where $\ov{a}$ is as in (\ref{defova}). The summands indexed by $j$ in the r.h.s.\ of (\ref{allez}) are identically distributed (note that due to lattice effects, this would not be true without the condition $kx \in \Z^d$). The usual ergodic theorem thus ensures that
$$
\frac1m \sum_{j = 0}^{m-1} \ov{a}_\lambda\big(jkx,(j+1)kx\big) \xrightarrow[m \to + \infty]{\text{a.s.}} \E[\ov{a}_\lambda(0,kx)],
$$
from which it follows that
$$
\frac{1}{n} \sum_{j = \lceil n/k \rceil}^{\lfloor \rho_2 n/k \rfloor-1} \ov{a}_\lambda\big(jkx,(j+1)kx\big) \xrightarrow[n \to + \infty]{\text{a.s.}} \frac{\rho_2 - 1}{k} \ \E[\ov{a}_\lambda(0,kx)].
$$
Besides, we can use Proposition~\ref{integrability} to bound the last two terms in the r.h.s.\ of (\ref{allez}). For the first one,
$$
\ov{a}_\lambda\big(nx, \lceil n/k \rceil kx) \le \sum_{j = 0}^{k} v_\lambda(nx+j),
$$
and a similar inequality holds for the second one. These observations yield, using (\ref{alltriv}) again,
$$
\limsup_{n \to + \infty} \frac1n \hata_\lambda(nx, \rho_2 nx) \le \frac{\rho_2 - 1}{k} \ \E[\ov{a}_\lambda(0,kx)].
$$
The conclusion then follows from (\ref{alphalambdainf}).

For the third claim, note that using inequality (\ref{firststepsubad}) with $x'=x_n\in \Delta^g(\rho_1 nx)$, $y = \rho_1 nx$ and $z = \rho_2 nx$, it comes that
$$
\hata_\lambda(\rho_1nx,\rho_2nx) \le \td{a}_\lambda(x_n,\rho_2 nx) \le \hata_\lambda(\rho_1 nx,\rho_2 nx) + v_\lambda(\rho_1nx),
$$
hence (\ref{alllambda2}) is a consequence of (\ref{alltriv}) and (\ref{alllambda1}).

We now need to see that with probability one, the three claims of the proposition hold uniformly over $\lambda$. The claims hold with probability one for all rational values of $\lambda$. To conclude, it suffices to note that the quantities under the limits in the l.h.s.\ of (\ref{alltriv}), (\ref{alllambda1}) and (\ref{alllambda2}) are monotone in $\lambda$, while the members of the r.h.s.\ are continuous in $\lambda$.
\end{proof}

\begin{prop}
\label{velocity}
With probability one, for any $\lambda > 0$, any $x \in \Q^d$ and any $s_1,s_2 \in \R_+$ such that
\begin{equation}
\label{cond:velocity}
(s_1,s_2) \cap [\alpha'_{\lambda+}(x),\alpha'_{\lambda-}(x)] \neq \emptyset,
\end{equation}
one has
\begin{equation}
\label{eq:velocity}
\lim_{n \to +\infty} \frac{1}{n} \hata_\lambda(0,nx,ns_1,ns_2) = \alpha_\lambda(x),
\end{equation}
and the same is true if $\hata_\lambda(0,nx,ns_1,ns_2)$ is replaced by $\td{a}_\lambda(x',nx,ns_1,ns_2)$, for any $x' \in \Delta^g(0)$.
\end{prop}
\begin{rem}
The proposition can be interpreted as saying that under condition (\ref{cond:velocity}), which boils down to a simple $s_1 < \alpha'_\lambda(x) < s_2$ if the derivative exists, there is no additional cost in asking the velocity to be in the interval $[s_1,s_2]$ when travelling in the direction of $x$. 
\end{rem}
\begin{proof}
Throughout this proof, we always place ourselves on the set of full probability measure on which Proposition~\ref{alllambda} holds.
Since $\hata_\lambda(0,nx,ns_1,ns_2) \ge \hata_\lambda(0,nx)$, we have
$$
\liminf_{n \to +\infty} \frac{1}{n} \hata_\lambda(0,nx,ns_1,ns_2) \ge \alpha_\lambda(x).
$$
The proof of the converse inequality proceeds in several steps.

\noindent \emph{Step 1.} We first show that the proposition is true under the stronger assumption that $(s_1,s_2)$ entirely contains $[\alpha'_{\lambda+}(x),\alpha'_{\lambda-}(x)]$. Observing that, for any $x' \in \Delta^g(0)$, 
\begin{eqnarray*}
e^{-\td{a}_\lambda(x',nx,ns_2,\infty)} & = & \EE_{x'} \left[ \exp\left( -\sum_{k = 0}^{\tau_g(nx) - 1} (\lambda + V)(S_k) \right), \mcl{V}_{nx,ns_2,\infty} \right] \\
& \le & \EE_{x'} \left[ \exp\left( -\sum_{k = 0}^{\tau_g(nx) - 1} (\lambda - \mu + V)(S_k) \right)e^{-\mu \tau_g(nx)}, \mcl{V}_{nx,ns_2,\infty} \right] \\
& \le & e^{- n \mu s_2} e^{-\td{a}_{\lambda - \mu}(x',nx)}
\end{eqnarray*}
and using part (3) of Proposition~\ref{alllambda}, we derive that
\begin{equation}
\label{v2infty}
\liminf_{n \to +\infty} \frac{1}{n} \td{a}_\lambda(x',nx,n s_2,+\infty) \ge  \alpha_{\lambda - \mu}(x) + \mu s_2.	
\end{equation}
If $s_2 > \alpha'_{\lambda-}(x)$, we can find $\mu$ small enough such that the r.h.s.\ of (\ref{v2infty}) is strictly larger than $\alpha_\lambda(x)$. Similarly, one can show that
$$
\liminf_{n \to + \infty} \frac{1}{n} \td{a}_\lambda(0,nx,0,ns_1) \ge \alpha_{\lambda + \mu}(x) - \mu s_1,
$$
which can be made strictly larger than $\alpha_\lambda(x)$ if $\mu$ is small enough since $s_1 < \alpha'_{\lambda+}(x)$.
Noting finally that
$$
e^{-\td{a}_\lambda(x',nx)} \le e^{-\td{a}_\lambda(x',nx,0,ns_1)} + e^{-\td{a}_\lambda(x',nx,ns_1,ns_2)} + e^{-\td{a}_\lambda(x',nx,ns_2,+\infty)},
$$
we obtain that for any $x' \in \Delta^g(0)$, 
$$
\limsup_{n \to + \infty} \frac{1}{n}\td{a}_\lambda(x',nx,ns_1,ns_2) \le \alpha_\lambda(x),
$$
and this implies (\ref{eq:velocity}). Note that in particular, the proposition is proved if $\alpha'_{\lambda+}(x) = \alpha'_{\lambda-}(x)$.

\noindent \emph{Step 2.} In this step, we prove that, for any $x',y,z \in \Z^d$ and any $s_1',s_1'', s_2',s_2''\ge 0$ such that 
\begin{equation}
\label{equalitiesons}
s_1'+s_1'' \ge s_1 \text { and } s_2'+s_2'' \le s_2,	
\end{equation}
one has
\begin{equation}
\label{subaddspeed1}
\td{a}_\lambda(x',z,s_1,s_2)  \le \td{a}_\lambda\big(x',y,s_1', s_2'\big) + \hata_\lambda\big(y, z, s_1'',  s_2'' - |\Delta^g(y)|\big) + v_\lambda(y).
\end{equation}
This is similar to equation~(\ref{firststepsubad}), but ``taking the clock into account''. Note that
$$
\1_{\mcl{V}_{z,s_1,s_2}} \ge \1_{\mcl{V}_{y,s_1', s_2'}} \ll( \1_{\mcl{V}_{z,s_1'',s_2''}} \circ \Theta_{\tau_g(y)} \rr).
$$
In words, this means that if one reaches $\Delta^g(y)$ after a number of steps in $[s_1',s_2']$ and then reaches site $\Delta^g(z)$ after a number of steps in $[s_1'',s_2'']$, then one has travelled from $x'$ to $\Delta^g(z)$ in a number of steps contained in $[s_1,s_2]$. Following the proof of Proposition~\ref{p:triangle}, we can use this observation and apply the Markov property at time $\tau_g(y)$. To conclude, we need to see that for any $y' \in \Delta^g(y)$, 
\begin{multline}
\label{arf}
\EE_{y'}\ll[ \exp\ll(- \sum_{n = 0}^{\tau_g(z) - 1} V_\lambda(S_n) \rr) \1_{\mcl{V}_{z,s_1'',s_2''}} \rr] \\
\ge \exp\ll(-v_\lambda(y) - \hata_\lambda\big(y, z, s_1'', s_2'' - |\Delta^g(y)|\big) \rr).
\end{multline}
This can be shown to be true with the same argument as in the proof of Proposition~\ref{p:triangle}: we choose $\ov{y} \in \Delta^g(y)$ such that $\td{a}(\ov{y},z) = \hata(y,z)$, and $\gamma$ a simple path connecting $y'$ to $\ov{y}$ inside $\Delta^g(y)$. We then restrict the expectation on the l.h.s.\ of (\ref{arf}) to the event ``$S = \gamma$'', and apply the Markov property at the time when the walk visits $\ov{y}$. The walk does not make more than $|\Delta^g(y)|$ steps before reaching $\ov{y}$, and the result follows.

\noindent \emph{Step 3.} Let $\rho \in (0,1) \cap \Q$ and $\eta > 0$ be such that
$$
\rho \alpha'_{\lambda-}(x) + (1-\rho) \alpha'_{\lambda+}(x) + [-\eta,\eta] \subset (s_1,s_2).
$$
We apply inequality (\ref{subaddspeed1}) replacing $(y,z,s_1,s_1',s_1'',s_2,s_2',s_2'')$ by $(\rho n x,nx,ns_1$, $ns_1'$, $ns_1'',ns_2,ns_2',ns_2'')$, where 
\begin{equation*}
s_1' = \rho \alpha'_{\lambda-}(x), \quad s_2' = \rho(\alpha'_{\lambda-}(x) + \eta),
\end{equation*}
\begin{equation*}
s_1'' = (1-\rho) (\alpha'_{\lambda+}(x)-\eta), \quad  s_2'' = (1-\rho)( \alpha'_{\lambda+}(x)+\eta),
\end{equation*}
and minimize it over $x \in \Delta^g(0)$, thus obtaining
\begin{multline}
\label{subaddspeed}
\hata_\lambda(0,nx,ns_1,ns_2)  \le \hata_\lambda\big(0,\rho nx,n s_1', n s_2'\big)\\
 + \hata_\lambda\big(\rho n x, nx, n s_1'', n s_2'' - |\Delta^g(\rho n x)|\big) + v_\lambda(\rho n x).
\end{multline}
Using part (1) of Proposition~\ref{alllambda}, we readily know that
\begin{equation}
\label{ingred1}
\lim_{n \to + \infty} \frac{v_\lambda(\rho n x)}{n} = 0.	
\end{equation}
The next two steps are devoted to the analysis of the first two summands in the r.h.s.\ of (\ref{subaddspeed}). 

\noindent \emph{Step 4.} Let $\lambda_1 < \lambda$ be such that $\alpha'_{\lambda_1}(x)$ exists and satisfies $\alpha'_{\lambda_1}(x) - \alpha'_{\lambda-}(x) \in [0,\eta]$. Arguing as in step 1, one can see that
$$
\limsup_{n \to + \infty} \frac1n \hata_\lambda\big(0,\rho nx,n s_1', n s_2'\big) \le \limsup_{n \to + \infty} \frac1n  \hata_{\lambda_1}\big(0,\rho nx,n s_1', n s_2'\big) + s_2' (\lambda - \lambda_1).
$$
Due to our choice of $\lambda_1$, one has indeed $\alpha'_{\lambda_1}(\rho x) \in (s_1',s_2')$, so one can apply the result obtained in step 1 and get
$$
\limsup_{n \to + \infty} \frac1n \hata_\lambda\big(0,\rho nx,\rho n s_1, \rho n s_2\big) \le  \alpha_{\lambda_1}(\rho x) + s_2' (\lambda - \lambda_1) \le \rho \alpha_{\lambda}(x) + \rho s_2' (\lambda - \lambda_1).
$$
This identity being valid for values of $\lambda_1$ that can be made arbitrarily close to $\lambda$, we have shown that
\begin{equation}
\label{ingred2}
\limsup_{n \to + \infty} \frac1n \hata_\lambda\big(0,\rho nx,\rho n s_1, \rho n s_2\big) \le \rho \alpha_{\lambda}(x).	
\end{equation}

\noindent \emph{Step 5.} We now turn to the second summand in the r.h.s.\ of (\ref{subaddspeed}). Using Proposition~\ref{tailradius} (and recalling that $|\Delta^g(z)| \le (3N/2+1)^d |\ovmclC_{i(z)}|$) , one can see that there exists $C > 0$ such that with probability one, for all but a finite number of $z \in \Z^d$, one has
$$
|\Delta^g(z)| \le C \ll(\log(\|z\|_1)\rr)^{d/(d-1)}.
$$
This implies in particular that
\begin{multline}
\label{ingred30}
\limsup_{n \to + \infty} \frac1n \hata_\lambda\big(\rho n x, nx, n s_1'', n s_2'' - |\Delta^g(\rho n x)|\big) \\
 \le \limsup_{n \to + \infty} \frac1n \hata_\lambda\big(\rho n x, nx, n s_1'', n s_2'''\big) ,
\end{multline}
where $s_2''' = (1-\rho) \alpha'_{\lambda+}(x) < s_2''$.
Let $\lambda_2 > \lambda$ be such that $\alpha'_{\lambda_2}(x)$ exists and satisfies $\alpha'_{\lambda_2}(x) - \alpha'_{\lambda+}(x) \in [-\eta,0]$. The r.h.s.\ of (\ref{ingred30}) is smaller than
$$
\limsup_{n \to + \infty} \frac1n \hata_{\lambda_2}\big(\rho n x, nx, n s_1'', n s_2'''\big).
$$
Note that $\alpha'_{\lambda_2}((1-\rho)x) \in (s_1'',s_2''')$. Reasoning as in step 1 with the help of part~(3) of Proposition~\ref{alllambda}, one obtains that this limsup is smaller than $(1-\rho) \alpha_{\lambda_2}(x)$. But this holds for values of $\lambda_2$ that can be arbitrarily close to $\lambda$, so we have shown that
\begin{equation}
\label{ingred3}
\limsup_{n \to + \infty} \frac1n \hata_\lambda\big(\rho n x, nx, n s_1'', n s_2'' - |\Delta^g(\rho n x)|\big) \le (1-\rho) \alpha_{\lambda+}(x).
\end{equation}

\noindent \emph{Step 6.} The conclusion for $\hata_\lambda$ now follows by combining inequalities (\ref{subaddspeed}), (\ref{ingred1}), (\ref{ingred2}) and (\ref{ingred3}). To see that (\ref{eq:velocity}) is also true when $\hata_\lambda(0,nx,ns_1,ns_2)$ is replaced by $\td{a}_\lambda(x',nx,ns_1,ns_2)$ for $x' \in \Delta^g(0)$, it suffices to note that, due to inequality (\ref{p:triangle}), one has
$$
\hata_\lambda(0,nx,ns_1,ns_2) \le \td{a}_\lambda(x',nx,ns_1,ns_2) \le \hata_\lambda(0,nx,ns_1,ns_2) + v_\lambda(0).
$$
\end{proof}

We are now equipped to tackle the purpose of this section, that is, to prove the upper bound part of the large deviation principle.
\begin{proof}[Proof of (\ref{ldopen}) in Theorem~\ref{t:ldp}]
In view of part (3) of Proposition~\ref{ratefunction}, it suffices to show that on the event $0 \in \mathbf{C}_\infty$, for every $x \in \Q^d$ and every $r \in \Q$, $r > 0$,
\begin{equation}
\label{ldopenrational}
\limsup_{n \to + \infty} \ -\frac1n \log \PP_{n,V}[S_n \in n D(x,r)] \le I(x).
\end{equation}
\noindent \emph{Step 1.} Assuming that $\lambda > 0$ is such that $\alpha'_{\lambda-}(x) < 1$, we show that, on the event $0 \in \mathbf{C}_\infty$,
\begin{equation}
\label{conclstep1}
\limsup_{n \to + \infty} \ -\frac1n \log \PP_{n,V}[S_n \in n D(x,r)] \le \alpha_\lambda(x) - \lambda.	
\end{equation}
As $Z_{n,V} \le 1$,  we have for any $\lambda \ge 0$,
\begin{eqnarray*}
\PP_{n,V}[S_n \in n D(x,r)] & \ge & \EE_0\ll[ \exp\ll(-\sum_{k = 0}^{n-1} V(S_k) \rr) , S_n \in nD(x,r) \rr] \notag \\
& \ge & e^{\lambda n} \ \EE_0\ll[ \exp\ll(-\sum_{k = 0}^{n-1} (\lambda + V)(S_k) \rr) , S_n \in nD(x,r) \rr] ,
\end{eqnarray*}
hence
\begin{multline*}
\limsup_{n \to + \infty} \ -\frac1n \log \PP_{n,V}[S_n \in n D(x,r)] \\
\le \limsup_{n \to + \infty} \ - \frac1n \log \EE_0\ll[ \exp\ll(-\sum_{k = 0}^{n-1} (\lambda + V)(S_k) \rr) , S_n \in nD(x,r) \rr] - \lambda.
\end{multline*}
In order to prove that (\ref{conclstep1}) holds, it thus suffices to see that, on the event $0 \in \mathbf{C}_\infty$,
\begin{equation}
\label{ldopen1}
\limsup_{n \to + \infty} \ - \frac1n \log \EE_0\ll[ \exp\ll(-\sum_{k = 0}^{n-1} V_\lambda(S_k) \rr) , S_n \in nD(x,r) \rr] \le \alpha_\lambda(x).
\end{equation}
On the event $0 \in \mathbf{C}_\infty$, there exists a livable simple path $\gamma$ connecting $0$ to a point $x' \in \Delta^g(0)$. Restricting the expectation below on the event ``$S = \gamma$'' and using the Markov property, we obtain
\begin{multline*}
\EE_0\ll[ \exp\ll(-\sum_{k = 0}^{n-1} V_\lambda(S_k) \rr) , S_n \in nD(x,r) \rr] \ge \exp\ll(-\sum_{z \in \gamma \setminus \{x'\}} (V_\lambda(z)+\log(2d)) \rr)\\ 
\EE_{x'}\ll[ \exp\ll(-\sum_{k = 0}^{n-|\gamma|-1} V_\lambda(S_k) \rr) , S_{n-|\gamma|} \in nD(x,r) \rr].
\end{multline*}
The first term in the r.h.s.\ above does not depend on $n$, so is irrelevant for our asymptotic analysis. For the second term, for $n$ sufficiently large, it is larger than
\begin{equation}
\label{step1.1}
\EE_{x'}\ll[ \exp\ll(-\sum_{k = 0}^{sn-1} V_\lambda(S_k) \rr) , S_{n-|\gamma|} \in nD(x,r) \rr],
\end{equation}
where $s$ is such that $0 \le \alpha'_{\lambda-}(x) < s < 1$. 

For every $y \in \Z^d$ for which it exists, let $\gamma_y$ be the shortest healthy path connecting $y$ to an $(\eps,R)$-clearing, with some deterministic tie-breaking rule, and let $c(y)$ be the clearing that is reached by $\gamma_y$. Proposition~\ref{clearings} ensures that
\begin{equation}
\label{ldopenallbutafew}
\begin{array}{l}
\text{for all but a finite number of }n\text{'s, one has} \\
\forall y \in \Delta^g(nx), \quad \gamma_y \cup B(c(y),R) \subset n D(x,r).
\end{array}	
\end{equation}
For trajectories with a starting point $S_0$ such that $\gamma_{S_0}$ is well defined, and for any $t \ge 0$, let $\mcl{E}_{t}$ be the event
$$
S = \gamma_{S_0} \text{ and } \forall k \text{ s.t. } |\gamma_{S_0}|\le k \le t, S_k \in B(c(S_0),R).
$$ 
In words, the event $\mcl{E}_t$ requires the walk to travel through the shortest possible healthy path to a clearing, and the stay within distance $R$ from that clearing up to time $n$.

Due to (\ref{ldopenallbutafew}), except possibly for a finite number of $n$'s, one has 
$$
\1_{\{S_{n-|\gamma|} \in nD(x,r)\}} \ge \1_{\{\tau_g(nx) \le s n\}} \ll(\1_{\mcl{E}_{n}} \circ \Theta_{\tau_g(nx)}\rr).
$$
That is to say, in order for the walk to be in $nD(x,r)$ at time $n - |\gamma|$, it suffices to reach a point of $\Delta^g(nx)$ before time $sn$ and from there, spend $n$ more steps going to the nearest reachable clearing and staying within distance $R$ from that clearing. We can use this observation and the Markov property to lower-bound the expectation in (\ref{step1.1}) by
$$
\EE_{x'}\ll[\exp\ll(-\sum_{k = 0}^{sn-1} V_\lambda(S_k) \rr) \1_{\{\tau_g(nx) \le s n\}} \PP_{S_{\tau_g(nx)}}[\mcl{E}_n] \rr].
$$
For any $z$ such that $\gamma_z$ is well defined, one has
\begin{eqnarray*}
\PP_{z}[\mcl{E}_n] & \ge & \exp\ll(-\sum_{z' \in \gamma_z \setminus \{c(z)\}} (V_\lambda(z') + \log(2d))\rr) e^{-\eps n} \ \PP_0[\forall k \le n, S_k \in B(0,R)] \\
& \ge & \exp\big(-(M+\lambda + \log(2d))|\gamma_z|\big) e^{-\eps n} \ \PP_0[\forall k \le n, S_k \in B(0,R)],
\end{eqnarray*}
where we used the fact that the path $\gamma_z$ visits only healthy sites. Using also Proposition~\ref{clearings} to bound the length of $\gamma_z$, and (\ref{restela}), we obtain that
$$
\limsup \ - \frac1n \min_{z \in \Delta^g(nx)}\PP_{z}[\mcl{E}_n] \le \eps - \frac{1}{R} \log \PP_0[S_R = 0].
$$
Since the probability $\PP_0[S_R = 0]$ decays only polynomially with $R$, and $\eps,R$ are arbitrary, claim (\ref{ldopen1}) (and thus also (\ref{conclstep1})) will be proved if we show that
\begin{equation}
\label{ldopen2}
\limsup_{n \to +\infty} -\frac1n \log \EE_{x'}\ll[\exp\ll(-\sum_{k = 0}^{sn-1} V_\lambda(S_k) \rr) \1_{\{\tau_g(nx) \le s n\}} \rr] \le \alpha_\lambda(x).
\end{equation}
The l.h.s.\ of (\ref{ldopen2}) is precisely
$$
\limsup_{n \to + \infty} \frac1n \td{a}(x',nx,0,sn).
$$
Since we assume $\lambda > 0$ and $\alpha'_{\lambda-}(x) < s$, Proposition~\ref{velocity} applies with $s_1 = 0$ and $s_2 = s$, and proves the announced claim.

\noindent \emph{Step 2.} 
If $\alpha'_{\lambda-}(x) < 1$ for any $\lambda > 0$, then $I(x) = \alpha_0(x)$, and the result is obtained applying (\ref{conclstep1}) along a sequence of $\lambda$'s tending to $0$. Otherwise, let 
$$
\lambda_0(x) = \sup \{\lambda > 0 :  \alpha'_{\lambda-} (x) \ge 1\} \in (0,+\infty].
$$
If $\lambda_0(x)$ is finite, then $I(x) = \alpha_{\lambda_0(x)}(x) - \lambda_0(x)$, and similarly, one obtains the results applying (\ref{conclstep1}) on a sequence of $\lambda$'s approaching $\lambda_0(x)$ from above. On the other hand, if 
\begin{equation}
\label{lambda0lim}
\lim_{\lambda \to +\infty} \alpha'_{\lambda-}(x) > 1,
\end{equation}
then one can easily check that $I(x) = +\infty$, so there is nothing to prove. It may well be however that $\lambda_0(x)$ is infinite while (\ref{lambda0lim}) fails to hold, in which case the l.h.s.\ of (\ref{lambda0lim}) is equal to $1$. Let $(\mu_k)_{k \in \N}$ be a sequence such that $\mu_k < 1$ and $\lim_{k \to \infty} \mu_k = 1$. For large enough $k$, one has
\begin{equation}
\label{comparerrsur2}
\PP_{n,V}[S_n \in n D(x,r)] \ge \PP_{n,V}[S_n \in n D(\mu_k x,r/2)].	
\end{equation}
Besides, the homogeneity of the norm $\alpha_\lambda$ guarantees that
$$
\lim_{\lambda \to +\infty} \alpha'_{\lambda-}(\mu_k x) = \mu_k < 1,
$$
so $\lambda_0(\mu_k x)$ is finite, in which case we know that
$$
\limsup_{n \to +\infty}\ -\frac1n \PP_{n,V}[S_n \in n D(\mu_k x,r/2)] \le I(\mu_k x).
$$
From the argument we used to prove part (3) of Proposition~\ref{ratefunction} (see (\ref{limhull})), we also learn that 
\begin{equation*}
I(x) = \lim_{k \to +\infty} I(\mu_k x),
\end{equation*}
and the results follows using (\ref{comparerrsur2}).
\end{proof}

%
%
%
%
%
%
%
%
\section{First-passage percolation}
\label{s:firstpass}
\setcounter{equation}{0}
In this last section, we briefly discuss corresponding results in the context of first-passage percolation. With randomness attached to sites, this means, instead of (\ref{defa}), to define $a(x,y)$ as
$$
a(x,y) = \inf \left\{ \sum_{z \in \gamma \setminus \{y\}} V(z) \ | \ \gamma \text{ n.n. path linking } x \text{ to } y \right\}.
$$
This corresponds to the zero-temperature limit of the model we have been considering, see \cite[Proposition~9]{zer}. 

In this case, Theorem~\ref{t:existpoint} and its proof are unchanged, except that in some cases, $\alpha$ may fail to be a norm, for there may exist $x \neq 0$ such that $\alpha(x) = 0$. Since $\alpha$ inherits the symmetries of the distribution of the potential, it can be seen that this implies $\alpha = 0$. Using classical percolation estimates, one can see that $\alpha = 0$ when $\P[V(0) = 0] > p_c$, and is non-zero when $\P[V(0) = 0] < p_c$. It should be possible to adapt the arguments of \cite{ck} to prove that, as a function of the distribution of $V(0)$, the semi-norm $\alpha$ is continuous on the open set of distributions such that $\P[V(0) = +\infty] < p_c$. This would ensure that $\alpha = 0$ if and only if $\P[V(0) = 0] \ge p_c$.

Provided $\alpha$ is indeed a norm, Theorems~\ref{t:shape} and \ref{t:existplane} can be kept unchanged. When $\alpha = 0$, equation (\ref{eq:shapemoment}) should be changed to
$$
\forall R > 0, \quad \P[D(0,R) \subset  t^{-1} A_t^\circ \text{ for all } t \text{ large enough}] = 1,
$$
equation (\ref{shapemeas}) to
$$
\forall R > 0, \quad |D(0,R) \setminus (t^{-1} A_t^\circ)  | \xrightarrow[t \to + \infty]{\text{a.s.}} 0,
$$
and equation (\ref{eq:pth2}) to
$$
\forall R > 0, \quad |D(0,R) \setminus (t^{-1} A_t^{\e_t})  | \xrightarrow[t \to + \infty]{\text{a.s.}} 0,
$$
while equation (\ref{eq:pth1}) is left without counterpart. In Theorem~\ref{t:existplane}, the limit in the r.h.s.\ of (\ref{eq:pth}) should be $0$.

Usually, first-passage percolation is defined with randomness attached to the edges of $\Z^d$. In this case, as was already mentioned, versions similar to Theorem~\ref{t:existpoint}, parts (1-2) of Theorem~\ref{t:shape} and Theorem~\ref{t:existplane} were obtained in \cite{cd,kes} (to be precise, \cite{cd} proved the results for $d = 2$, while \cite{kes} considered any $d \ge 2$, but for the purpose of the lecture notes, assumed finite second moment of the passage times in his proof of almost sure convergence of the point-to-point exponent). The only difference appearing when randomness is attached to edges is that conditions for a.s.\ and $L^1$ convergence of the point-to-point exponent become the same, and that the approximate $a_m(x,y)$ we introduced is no longer necessary. Part (3) of Theorem~\ref{t:shape} and Theorem~\ref{t:existplane} remain valid in this context (provided $\alpha$ is a norm, otherwise the statements should be understood as discussed above).

\medskip

\noindent \textbf{Acknowledgments.} I would like to thank the two referees for their careful reviews, and in particular for pointing out an error in the initial proof of Proposition~\ref{controltailH}.

\end{document}